\documentclass[12pt]{article}
\usepackage{amsmath}
\usepackage{amsfonts}
\usepackage{amssymb}
\usepackage{graphicx}

\newtheorem{theorem}{Theorem}[section]

\newtheorem{corollary}[theorem]{Corollary}

\newtheorem{definition}[theorem]{Definition}

\newtheorem{lemma}[theorem]{Lemma}

\newtheorem{proposition}[theorem]{Proposition}

\newenvironment{proof}[1][Proof]{\noindent\textbf{#1.} }{\ \rule{0.5em}{0.5em}}
\input{tcilatex}

\begin{document}

\title{The dependence on parameters of the inverse functor to the $K$%
--finite functor }
\author{Nolan R. Wallach}
\date{}
\maketitle

\begin{abstract}
An interpretation of the Casselman-Wallach (C-W) Theorem is that the $K$%
--finite functor is an isomorphism of categories from the category of
finitely generated, admissible smooth Fr\'{e}chet modules of moderate growth
to the category of Harish-Chandra modules for a real reductive group, $G$
(here $K$ is a maximal compact subgroup of $G$).In this paper we study the
dependence of this functor on parameters. Our main result implies that
holomorphic dependence implies holomorphic dependence. The work uses results
from the excellent thesis of van der Noort. Also a remarkable family of \
universal Harish-Chandra modules developed in this paper plays a key role.
\end{abstract}

\section*{Introduction}

An interpretation of the Casselman-Wallach (C-W) Theorem is that the $K$%
--finite functor is an isomorphism of categories from the category of
finitely generated, admissible smooth Fr\'{e}chet modules of moderate growth
to the category of Harish-Chandra modules for a real reductive group, $G$
(here $K$ is a maximal compact subgroup of $G$). This variant will be
explained in detail in the next section. Also, the inverse functor, $%
V\rightarrow\overline{V\text{,}}$ to the $K$--finite functor is described
therein. In this paper we study the dependence of this functor on
parameters. Our main result implies that holomorphic dependence implies
holomorphic dependence (see Theorem \ref{main}). This work rests on the
excellent thesis of Vincent van der Noort (\cite{VanderNoort}) which
contains several remarkable theorems including his finiteness theorem that
is given a slightly simplified proof in Appendix \ref{VanderN} for the
benefit of the reader. In addition to the work of van der Noort our
technique involves the study of a class of standard modules in the
Harish-Chandra category with remarkable properties. In particular, they are
free modules for the universal enveloping algebra of a maximal unipotent
subgroup of $G$. Also, every Harish-Chandra module has a resolution by these
modules.

The technical general results not specific to the main results of this paper
are the content of the many appendices to this paper. In particular, several
of the appendices involve the study of continuous families of Hilbert
representations. The first step of the proof of the main theorem is to show
that locally an analytic family of Harish-Chandra modules can be globalized
to a continuous family of (strongly continuous) Hilbert representations.

One can ask if the work of Bernstein and Kr\"{o}tz \cite{BerKro} contains
our main result? To be blunt, it is not clear what exactly they meant by
dependence on parameters. Their proof of the C-W theorem for linear real
reductive groups contains a proof of the automatic continuity theorem that
takes into account the parameters of the principal series . Their proof of
the full C-W theorem is identical in its last stages to the one in \cite%
{RRG1-11} involving an ingenious argument due to Casselman that doesn't take
into account dependence on parameters. In any event, since only the original
C-W theorem is used in this paper their version of the result would serve
equally well as a basis for the proof.

The appendices take up more space than the body of the paper. Hopefully this
separation will help the reader see the flow of the proof of the main
theorem. Among the appendices there are results that are of interest beyond
this paper. For example, Appendix A gives a proof that the $C^{\infty}$
vectors relative to $G$ of a finitely generated, admissible Hilbert
representation are the same as the $C^{\infty}$ vectors relative to $K$ (see
Proposition \ref{C-K}). Also, as mentioned above, Appendix \ref{VanderN} \
contains a proof of an important result of van der Noort.

The proof of the main theorem follows the following lines. First a class of
Harish-Chandra modules is constructed which we call $J$--modules (for lack
of a name) that have the property that every Harish-Chandra module has a
resolution by $J$-modules. It is shown that if one has an analytic family of
Harish-Chandra modules\ref{an-gK} and if $U$ is an open set with compact
closure in the parameter space there is a family of $J$--modules over $U$
mapping surjectively onto the restriction of the family to $U$. The next
step is to locally (in the parameter) globalize a continuous family of $J$%
-modules to a continuous family of Hilbert representations satisfying a
technical condition (smoothable) that implies that the corresponding family
of $C^{\infty}$--vectors defines a continuous family of smooth Fr\'{e}chet
representations of locally uniform moderate growth (in the parameter). The
last stage is to start with a holomorphic family of Harish--Chandra modules
use the Hilbert modules corresponding to the resolving $J$--module to
construct a local (in the parameter) Hilbert globalizations of the the
family satisfying the accessibility condition. \ The fact that the C-W
theorem is an isomorphism of categories shows that the corresponding local
families of smooth Fr\'{e}chet representations \textquotedblleft
glue\textquotedblright\ together and complete the proof of the main theorem.

\section{The C-W Theorem}

Let $G$ be a real reductive group and let $K$ be a maximal compact subgroup
of $G$. Throughout the paper, if $H$ is a Lie group over $\mathbb{R}$ then
its (real) Lie algebra will be denoted $\mathfrak{h}_{o}$ (i.e lower case
fractur $H$ sub-$o$) and its complexification denoted $\mathfrak{h}$. Let $%
\theta$ denote the Cartan involution of $G$ (and of $\mathfrak{g}_{o}$)
corresponding to $K$. Set $\mathfrak{k}_{o}=Lie(K)$, $\mathfrak{t=t}%
_{o}\otimes\mathbb{C}$and $\mathfrak{p}_{o}=\{X\in\mathfrak{g}|\theta X=-X\}$
Fix a symmetric $Ad(G)$--invariant bilinear form, $B$, on $\mathfrak{g}_{o}$
such that $B_{|\mathfrak{k}_{o}}$ is negative definite and $B_{|\mathfrak{p}%
_{o}}$ is positive definite. Let $v_{1},...,v_{n}$ be an orthonormal basis
of $\mathfrak{g}$ with respect to $B$ and set $C=\sum v_{i}^{2}$, the
corresponding Casimir operator). Let $C_{K}$ be the Casimir operator for $%
\mathfrak{k}$ corresponding to $B_{|\mathfrak{k}}$.Let $\mathcal{H}(%
\mathfrak{g},K)$ denote the category of Harish-Chandra modules, that is, the
finitely generated, admissible, $(\mathfrak{g},K)$--modules. The other
category of representations that appear in the C-W theorem is the category $%
\mathcal{HF}(G)$ of admissible finitely generated smooth Fr\'{e}chet
representations of moderate growth. An object in $\mathcal{HF}(G)$ is a pair 
$(\pi,V)$ with $V$ a Fr\'{e}chet space and $\pi$ a homomorphism of $G$ into
the group of continuous bijections of $V$ such the the following 3
conditions are satisfied

1. The map $G\times V\rightarrow V$ given by $g,v\mapsto\pi(g)v$ is
continuous and is $C^{\infty}$ in $G$.

2. Let $\left\Vert ...\right\Vert $ be a norm on $G$ (see Appendix \ref%
{norms}). If $p$ is a continuos seminorm on $V$ then there exists $q$ a
continuous seminorm on $V$ and $r$ such that $p(\pi(g)v)\leq\left\Vert
g\right\Vert ^{r}q(g)$.

3. The representation of $\mathfrak{g}$ , $d\pi$, defines on the $K$--finite
vectors of $V,V_{K}$, an object in $\mathcal{H}(\mathfrak{g},K)$ .

One version of the C-W Theorem is (see \cite{RRG1-11})

\begin{theorem}
If $(\pi,V),(\mu,W)\in\mathcal{HF}(G)$ and $T:(d\pi,V_{K})\rightarrow
(d\sigma,W_{K})$ is a morphism in $\mathcal{H}(\mathfrak{g},K)$ then $T$
extends to a morphism in $\mathcal{HF}(G)$ with closed image that is a
topological summand.
\end{theorem}

Let $(\pi,V)\in\mathcal{H}(\mathfrak{g},K)$ then a $K$--invariant Hermitian
inner product on $V,\left\langle ...,...\right\rangle $, will be called $G$%
--integrable if there exists a strongly continuous action, $\sigma$, of $G$
on the Hilbert space completion of $V$ , $H_{\left\langle ,\right\rangle }$,
relative to $\left\langle ...,...\right\rangle ,$ such that the $(\mathfrak{g%
},K)$--module of $K$--finite $C^{\infty}$ vectors, $(d\sigma,$ $\left(
H_{\left\langle ,\right\rangle }\right) _{K}^{\infty}\ )=(\pi,V)$. The
subquotient theorem implies that there exists at least one $G$--integrable
inner product on $V$. Let $\mathcal{I}(\pi,V)$ be the set of integrable $K$%
--invariant inner products on $V$. If $\left\langle ...,...\right\rangle \in%
\mathcal{I}(\pi,V)$ then $\left( H_{\left\langle ....,...\right\rangle
}\right) ^{\infty}\in\mathcal{HF}(G)$.

The theorem implies that if $\left\langle ...,...\right\rangle _{i}\in%
\mathcal{I}(V),i=1,2$ then 
\begin{equation*}
\left( H_{\left\langle ...,...\right\rangle _{1}}\right) ^{\infty}=\left(
H_{\left\langle ...,...\right\rangle _{2}}\right) ^{\infty}. 
\end{equation*}
In particular this implies that the norm $v\mapsto\left\Vert v\right\Vert
_{2}$ is continuous on $\left( H_{\left\langle ...,...\right\rangle
_{1}}\right) ^{\infty}$. Proposition \ref{C-K} implies that there exists $C$
and $l$ such that%
\begin{equation*}
\left\Vert v\right\Vert _{2}\leq\left\Vert
d\sigma_{1}(g)(1+C_{K})^{i}v\right\Vert _{1}. 
\end{equation*}
Note that if $\left\langle ...,...\right\rangle \in\mathcal{I}(V)$ then the $%
K$--invariant inner product $\left\langle v,w\right\rangle _{1}=\left\langle
\pi(1+C_{k})^{l}v,w\right\rangle $ is also in $\mathcal{I}(V)$. This allows
us to define an inverse to the $K$--finite functor. Set%
\begin{equation*}
\overline{V}=\{\{v_{\gamma}\}\in\prod_{\gamma\in\widehat{K}}V(\gamma
)|\sum_{\gamma\in\widehat{K}}\left\langle v_{\gamma},v_{\gamma}\right\rangle
^{2}<\infty,\forall\left\langle ...,...\right\rangle \in\mathcal{I}(V)\}. 
\end{equation*}
Noting (as above) that this space is equal to $\left( H_{\left\langle
,\right\rangle }\right) ^{\infty}$ for any $\left\langle
...,...\right\rangle \in\mathcal{I}(V)$ the space $\overline{V}$ endowed
with the topology given by the norms $\left\{ \left\Vert ...\right\Vert
_{\left\langle ...,..\right\rangle }\right\} _{\left\langle
...,..\right\rangle \in\mathcal{I}(V)}$ is an object in $HF(V)$ with $%
\overline{V}_{K}=V$.

This leads to

\begin{theorem}
The functor $V\rightarrow V_{F}$ from $\mathcal{HF}(G)$ to $\mathcal{H}(%
\mathfrak{g},K)$ is an isomorphism of categories with inverse functor $%
V\rightarrow\overline{V}$.
\end{theorem}

The rest of this paper will be devoted to the study of the dependence of
this functor on parameters. For this we will use a class of universal
modules with remarkable properties related to ones in \cite{RRG1-11},Section
11.3 and in \cite{HOW}.

\section{The subalgebra \textbf{D} of $Z(\mathfrak{g)}$}

Let $G$ be a real reductive group of inner type. That is, if $\mathfrak{g}%
_{o}=Lie(G)$, $\mathfrak{g}=\mathfrak{g}_{o}\otimes\mathbb{C}$ then $Ad(G)$
is contained in the identity component of $Aut(\mathfrak{g}).$Keep the
notation of the previous section. let $p$ be the projection of $\mathfrak{g}$
onto $\mathfrak{p}=\mathfrak{p}_{o}\otimes\mathbb{C}$ corresponding to $%
\mathfrak{g}_{o}=\mathfrak{k}_{o}\oplus\mathfrak{p}_{o}$. Extend $p$ to a
homomorphism of $S(\mathfrak{g})$ onto $S(\mathfrak{p})$. Then $p$ is the
projection corresponding to%
\begin{equation*}
S(\mathfrak{g})=S(\mathfrak{p})\oplus S(\mathfrak{g})\mathfrak{k}. 
\end{equation*}
In \cite{HOW} we found homogeneous elements $w_{1},...,w_{l}$ of $S(%
\mathfrak{g})^{G}$ with $w_{1}=C$ Satisfying the following two properties:

1. $p(w_{1}),...,p(w_{l})$ are algebraically independent.

2. There exists a finite dimensional homogeneous subspace $E$ of $S(%
\mathfrak{p})^{K}$ such that the map $\mathbb{C[}p(w_{1}),...,p(w_{l})]%
\otimes E\rightarrow S(\mathfrak{p})^{K}$ given by multiplication is an
isomorphism.

If $\mathfrak{g}$ contains no simple ideals of type E (we will list the real
forms of type $E$ that need to be avoided later in this section) one can
take $E=\mathbb{C}1.$ If $\mathfrak{g}$ is split over $\mathbb{R}$ then $%
\mathbb{C}[w_{1},...,w_{l}]=S(\mathfrak{g})^{G},E=\mathbb{C}1$.

Let $\mathcal{H}$ denote the space of harmonic elements of $S(\mathfrak{p})$%
, that is, the orthogonal complement to the ideal $S(\mathfrak{p})\left( S(%
\mathfrak{p})\mathfrak{p}\right) ^{K}$ in $S(\mathfrak{p})$ relative to the
Hermitian extension of inner product $B_{|\mathfrak{p}_{o}}$. Then the
Kostant-Rallis theorem (\cite{KosRal}) implies that the map%
\begin{equation*}
\mathcal{H}\otimes S(\mathfrak{p})^{K}\rightarrow S(\mathfrak{p}) 
\end{equation*}
given by multiplication is a linear bijection. This and 2. easily imply

\begin{lemma}
The map 
\begin{equation*}
\mathcal{H}\otimes E\otimes\mathbb{C}[w_{1},...,w_{l}]\otimes S(\mathfrak{k}%
)\rightarrow S(\mathfrak{g}) 
\end{equation*}
given by multiplication is a linear bijection.
\end{lemma}

Let $\mathfrak{a}_{o}$ be a maximal abelian subspace of $\mathfrak{p}_{o}$
and let (as usual) 
\begin{equation*}
W=W(\mathfrak{a})=\{s\in GL(\mathfrak{a})|s=Ad(k)_{|\mathfrak{a}},k\in K\}. 
\end{equation*}
Let $h\in\mathfrak{a}_{o}$ be such that $\mathfrak{a}_{o}=\{X\in \mathfrak{p}%
_{o}|[h,X]=0\}$. If $\lambda\in\mathbb{R}$ then set $\mathfrak{g}%
_{o}^{\lambda}=\{X\in\mathfrak{g}_{o}|[h,X]=\lambda X\}$. Set $\mathfrak{n}%
_{o}\mathfrak{=\oplus}_{\lambda>0}\mathfrak{g}_{o}^{\lambda}$ and $\mathfrak{%
\bar{n}}_{o}\mathfrak{=\theta n}_{o}\mathfrak{=\oplus}_{\lambda >0}\mathfrak{%
g}_{o}^{-\lambda}$. Then%
\begin{equation*}
\mathfrak{p}=p(\mathfrak{n)\oplus a}
\end{equation*}
and $p(\mathfrak{n})$ is the orthogonal complement to $\mathfrak{a}$ in $%
\mathfrak{p}$ relative to $B$. Let $q$ be the projection of $\mathfrak{p}$
onto $\mathfrak{a}$ corresponding to this decomposition. Then the Chevalley
restriction theorem implies that 
\begin{equation*}
q:S(\mathfrak{p})^{K}\rightarrow S(\mathfrak{a})^{W}
\end{equation*}
is an isomorphism of algebras. Also, as above, if $H$ is the orthogonal
complement to $\left( S(\mathfrak{a)a}\right) ^{W}S(\mathfrak{a)}$ in $S(%
\mathfrak{a)}$. Then the map%
\begin{equation*}
S(\mathfrak{a})^{W}\otimes H\rightarrow S(\mathfrak{a}) 
\end{equation*}
given by multiplication is a linear bijection. Putting these observations
together the map%
\begin{equation*}
S(\mathfrak{n)\otimes}S(\mathfrak{a})^{W}\otimes H\otimes S(\mathfrak{k}%
)\rightarrow S(\mathfrak{g}) 
\end{equation*}
given by multiplication is a linear bijection. We also note that the map%
\begin{equation*}
\mathbb{C[}w_{1},...,w_{l}]\otimes E\rightarrow S(\mathfrak{a})^{W}
\end{equation*}
given by%
\begin{equation*}
w\otimes e\mapsto q(p(w))q(e) 
\end{equation*}
is a linear bijection. This in turn implies

\begin{lemma}
The map%
\begin{equation*}
S(\mathfrak{n)\otimes}\mathbb{C[}w_{1},...,w_{l}]\otimes E\otimes H\otimes S(%
\mathfrak{k})\rightarrow S(\mathfrak{g}) 
\end{equation*}
given by multiplication is a linear bijection.
\end{lemma}

Let $\mathrm{symm}$ denote the symmetrization map from $S(\mathfrak{g})$ to $%
U(\mathfrak{g})$ then $\mathrm{symm}$ is a linear bijection and $\mathrm{symm%
}\circ Ad(g)=Ad(g)\circ\mathrm{symm}$. Let $Z(\mathfrak{g})=U(\mathfrak{g}%
)^{G}$ denote the center of $U(\mathfrak{g})$. Set $z_{i}=\mathrm{symm}%
(w_{i})$ and 
\begin{equation*}
\mathbf{D}=\mathbb{C}[z_{1},...,z_{l}]. 
\end{equation*}
Note that if $S_{j}(\mathfrak{g})=\sum_{k\leq j}S^{j}(\mathfrak{g})$ and if $%
U^{j}(\mathfrak{g})\subset U^{j+1}(\mathfrak{g})$ is the standard filtration
of $U(\mathfrak{g}_{\mathbb{C}})$ then 
\begin{equation*}
\mathrm{symm}(S_{j}(\mathfrak{g}))=U^{j}(\mathfrak{g}). 
\end{equation*}
The above and standard arguments (\cite{HOW} Theorem 2.5 and Lemma 5.2) imply

\begin{theorem}
\label{decompositions}Let the notation be as above. Then

1. The map 
\begin{equation*}
\mathcal{H}\otimes E\otimes\mathbf{D}\otimes U(\mathfrak{k})\rightarrow U(%
\mathfrak{g}) 
\end{equation*}
given by%
\begin{equation*}
h\otimes e\otimes D\otimes k\mapsto\mathrm{symm}(h)\mathrm{symm}(e)Dk 
\end{equation*}
is a linear bijection.

2. The map%
\begin{equation*}
U(\mathfrak{n})\otimes E\otimes H\otimes\mathbf{D\otimes U(}\mathfrak{k}%
)\rightarrow U(\mathfrak{g}) 
\end{equation*}
given by%
\begin{equation*}
n\otimes e\otimes h\otimes D\otimes k\mapsto n\mathrm{symm}(e)\mathrm{symm}%
(h)Dk 
\end{equation*}
is a linear bijection.
\end{theorem}

\section{A class of admissible finitely generated $(\mathfrak{g},K)$--modules%
}

Retain the notation in the preceding section. Note that Theorem \ref%
{decompositions} implies that the subalgebra $\mathbf{D}U(\mathfrak{k})$of $%
U(\mathfrak{g})$ is isomorphic with the tensor product algebra $\mathbf{D}%
\otimes U(\mathfrak{k})$ and that $U(\mathfrak{g})$ is free as a right $%
\mathbf{D}U(\mathfrak{k})$--module under multiplication. If $R$ is a $%
\mathbf{D}U(\mathfrak{k})$--module then set 
\begin{equation*}
J(R)=U(\mathfrak{g})\otimes_{\mathbf{D}U(\mathfrak{k})}R. 
\end{equation*}

Denote by $\mathcal{H}(\mathfrak{g},K)$ the Harish--Chandra category of
admissible finitely generated $(\mathfrak{g},K)$--modules. Let $R$ be a
finite dimensional continuous $K$--module that is also a $\mathbf{D}$%
--module and the actions commute then $K$ acts on $J(R)$ as follows:%
\begin{equation*}
k\cdot\left( g\otimes r\right) =Ad(k)g\otimes kr,k\in K,g\in U(\mathfrak{g}%
),r\in R. 
\end{equation*}
Then as a $K$--module 
\begin{equation*}
J(R)\cong\mathcal{H}\otimes E\otimes R 
\end{equation*}
with $K$ acting trivially on $E$. Note that $J(R)\in\mathcal{H}(\mathfrak{g}%
,K)$ since the multiplicities of $K$--types in $\mathcal{H}$ are finite and $%
J(R)$ is clearly finitely generated as a $U(\mathfrak{g})$--module. Let $W(%
\mathbf{D,}K)$ be the category of finite dimensional $(\mathbf{D},K)$%
--modules with $K$ acting continuously and the action of $\mathbf{D}$ and $K$
commute.

\begin{lemma}
$R\rightarrow J(R)$ defines an exact faithful functor from the category $W(K,%
\mathbf{D})$ to $\mathcal{H}(\mathfrak{g},K)$.
\end{lemma}

\begin{proof}
This follows since $U(\mathfrak{g})$ is free as a module for $\mathbf{D}U(%
\mathfrak{k})$ under right multiplication.
\end{proof}

As usual, denote the set of equivalence classes of irreducible, finite
dimensional, continuous representations of $K$ by $\hat{K}$. If $V\in 
\mathcal{H}(\mathfrak{g},K)$ set $V(\gamma)$ equal to the sum of all
irreducible $K$--subrepresentations of $V$ in the class of $\gamma$. Then $%
V(\gamma)$ is invariant under the action of $Z(\mathfrak{g})$ hence under
the action of $\mathbf{D}$.

By definition if $V\in\mathcal{H}(\mathfrak{g},K)$ there is a finite subset $%
F\subset\hat{K}$ such that 
\begin{equation*}
U(\mathfrak{g}_{\mathbb{C}})\sum_{\gamma\in F}V(\gamma). 
\end{equation*}
Set $R=\sum_{\gamma\in F}V(\gamma)\in W(\mathbf{D,}K)$. One has a canonical $%
(\mathfrak{g},K)$--module surjection $J(R)\rightarrow V$ given by $g\otimes
r\mapsto gr.$ A submodule of an element of $\mathcal{H}(\mathfrak{g},K)$ is
in $\mathcal{H}(\mathfrak{g},K)$ so

\begin{proposition}
If $V\in\mathcal{H}(\mathfrak{g},K)$ then there exists a sequence of
elements $R_{j}\in W(\mathfrak{g},K)$ and an exact sequence in $\mathcal{H}(%
\mathfrak{g},K)$%
\begin{equation*}
...\rightarrow J(R_{k})\rightarrow....\rightarrow J(R_{2})\rightarrow
J(R_{1})\rightarrow J(R_{0})\rightarrow V\rightarrow0. 
\end{equation*}
\end{proposition}

Notice that this exact sequence us a free resolution of $V$ as a $U(%
\mathfrak{n})$--module.

Let $\beta:\mathbf{D}\rightarrow\mathbb{C}$ be an algebra homomorphism. Let $%
\mathcal{H}(\mathfrak{g},K)_{\beta}$ be the full subcategory of $\mathcal{H}(%
\mathfrak{g},K)$ consisting of modules $V$ such that if $z\in\mathbf{D}$
then it acts by $\beta(z)I$. The next result is an aside that will not be
used in the rest of this paper and is a simple consequence of the definition
of projective object.

\begin{lemma}
Let $F$ be a finite dimensional $K$--module and let $\mathbf{D}$ act on $F$
by $\beta(z)I$ yielding an object $R\in W(K,\mathbf{D})$. Then $J(R)$ is
projective in $\mathcal{H}(\mathfrak{g},K)_{\beta}$.
\end{lemma}

\section{The objects in $W(K,\mathbf{D)}$}

If $R\in W(K,\mathbf{D)}$ then $R$ has a $K$--isotypic decomposition $%
R=\oplus_{\gamma\in\hat{K}}R(\gamma)$. Only a finite number of the $R(\gamma)
$ are non-zero. If $D\in\mathbf{D\ }$then $DR(\gamma)\subset R(\gamma)$ for
all $\gamma\in\hat{K}$. If $\chi:\mathbf{D}\rightarrow \mathbb{C}$ is an
algebra homomorphism then we set $R_{\chi}=\{v\in R|(D-\chi(D))^{k}v=0,$for
some $k>0\}$ \ Then setting $ch(\mathbf{D})$ equal to the set of all algebra
homomorphisms of $\mathbf{D}$ to $\mathbb{C}$ we have the decomposition%
\begin{equation*}
R=\bigoplus_{\gamma\in\hat{K},\chi\in ch(\mathbf{D})}R_{\chi}(\gamma). 
\end{equation*}
Fix a $K$--module $(\tau_{\gamma},F_{\gamma})\in\gamma$. Then $%
R_{\chi}(\gamma)$ is isomorphic with 
\begin{equation*}
\mathrm{Hom}_{K}(V_{\gamma},R_{\chi})\otimes F_{\gamma}
\end{equation*}
with $K$ acting on $F_{\gamma}$ and $\mathbf{D}$ acting on $\mathrm{Hom}%
_{K}(V_{\gamma},R)$.

If $R$ is an irreducible object in $W(K,\mathbf{D)}$ then Schur's lemma
implies that $\mathbf{D}$ acts by a single homomorphism to $\mathbb{C}$ and $%
R$ is irreducible as a $K$--module. Set $V_{\gamma,\chi,}$ equal to the
module with $\mathbf{D}$ acting by $\chi$ and $K$ acting by an element of $%
\gamma$.

We next analyze the homomorphisms $\chi$. Let $\chi$ be such a homomorphism
then $\chi(z_{i})=\lambda_{i}\in\mathbb{C}$. Thus one simple parametrization
is by $(\lambda_{1}....,\lambda_{l})\in\mathbb{C}^{l}$. We use the notation $%
\beta_{\lambda}$ for the homomorphism such that $\beta_{\lambda}(z_{i})=%
\lambda_{i}$.

\begin{definition}
Let $X$ be an analytic manifold. An analytic family in $W(K,\mathbf{D)}$
over $X$ is a pair $(\mu,V)$ of a a finite dimensional continuous $K$%
--module,$V$, and a $\mu:X\times\mathbf{D}\rightarrow\mathrm{End}(V)$ such
that $D\mapsto \mu(x,D)$ is a representation of $\mathbf{D}$ on $V$ and $%
x\mapsto\mu(x,D)$ is analytic for all $D\in\mathbf{D}$.
\end{definition}

\section{Parabolically induced families}

Let $A$ and $N$ be the connected subgroups of $G$ with $Lie(A)=\mathfrak{a}%
_{o}$ and $Lie(N)=\mathfrak{n}_{o}$. Let $M$ be the centralizer of $%
\mathfrak{a}$ in $K$. Set $Q=MAN$ then $Q$ is a minimal parabolic subgroup
of $G$.

\begin{definition}
An analytic family of finite dimensional $Q$--modules over a real analytic
manifold $X$ is a pair $(\sigma,S)$ with $S$ a finite dimensional continuous 
$M$--module and a real analytic map $\sigma:X\times Q\rightarrow GL(S)$ such
that $x\mapsto\sigma(x,q)$ is analytic and $\sigma(x,\cdot)=\sigma_{x}$ is a
representation of $Q$.
\end{definition}

Let $(\sigma,S)$ be a continuous finite dimensional representation of $Q$.
Set $I^{\infty}(\sigma_{|M})$ equal to the space of all smooth functions $%
f:K\rightarrow S$ satisfying $f(mk)=\sigma(m)f(k)$. Define and action $%
\pi_{\sigma}$ of $G$ on $I^{\infty}(\sigma_{|M})$ as follows: if $f\in
I^{\infty}(\sigma_{|M})$ then extend $f$ to $G$ by $f_{\sigma}(qk)=\sigma
(q)f(k)$, then, since $K\cap Q=M$ and $QK=G$, $f_{\sigma}$ is $C^{\infty}$
on $G$ set $\pi_{\sigma}(g)f(k)=f_{\sigma}(kg)$. Also set%
\begin{equation*}
\pi_{\sigma}(Y)f(k)=\frac{d}{dt}f_{\sigma}(k\exp tY)_{|t=0}
\end{equation*}
for $Y\in\mathfrak{g}$ and $k\in K,f\in I^{\infty}(\sigma_{|M})$. Let $%
I(\sigma_{|M})$ be the space of all right $K$ finite elements of $I^{\infty
}(\sigma_{|M})$

Put and $M$--invariant inner product, $\left\langle ...,...\right\rangle $
on $S$. If $f,h\in I^{\infty}(\sigma_{|M})$ then set 
\begin{equation*}
(f,h)=\int_{K}\left\langle f(k),h(k)\right\rangle dk 
\end{equation*}
with $dk$ normalized invariant measure on $K$. The following is standard.

\begin{proposition}
Let $(\sigma,S)$ be an analytic family of finite dimensional representations
of $Q$ over the analytic manifold $X$. Set $\lambda(x,y)=\pi_{\sigma_{x}}(y)$
for $x\in X,y\in U(\mathfrak{g}_{\mathbb{C}})$. If $\mu$ is the common value
of $\sigma_{x}|_{M}$, then $(\lambda,I(\mu))$ is an analytic family (see
Appendix \ref{an-gK}) of objects in $\mathcal{H}(\mathfrak{g},K)$ over $X$.
\end{proposition}

\begin{proof}
It is standard that 
\begin{equation*}
x,g\mapsto(\pi_{\sigma_{x}}(g)f,h) 
\end{equation*}
is real analytic and holomorphic in $x$ for $f,h\in I(\mu)$.
\end{proof}

\begin{proposition}
\label{Par-access}Let $(\sigma,S)$ be an analytic family of $Q$--modules
based on $Z.$Set $\sigma(m)$ equal to the common value of $\sigma_{z}(m)$
for $m\in M$ and $H$ equal to the unitarily induced representation of $\sigma
$ from $M$ to $K.$ Then $z\rightarrow(\pi_{\sigma_{z}},H)$ is a continuous
family of Hilbert representations over $Z$ (see Definition \ref{Hilbert-Fam}%
) that is smoothable in the sense of Definition \ref{smoothable}.
\end{proposition}

\begin{proof}
\begin{equation*}
f(mk)=\sigma(m)f(k),m\in M,k\in K. 
\end{equation*}
Recall that $f_{\sigma_{x}}(g)=f_{\sigma_{x}}(namk)=\sigma_{x}(nam)f(k)$ for 
$g=namk,n\in N,a\in A,m\in M,k\in K$. Let $\{n_{1}.n_{2},....\},%
\{a_{1},a_{2},...\}$ be respectively bases of $U(\mathfrak{n})$ and $U(%
\mathfrak{a})$. Let $Y_{1},...,Y_{n}$ be a basis of $\mathfrak{k}_{o}$such
that $B(Y_{i},Y_{j})=-\delta_{ij}$. The monomials $Y^{I}=\
Y_{1}^{i_{1}}\cdots Y_{n}^{i_{n}}$ form a basis of $U(\mathfrak{k})$. If $%
u\in U(\mathfrak{g)}$ and if $f\in H^{\infty}$ then 
\begin{equation*}
d\pi_{\sigma_{z}}(u)f(k)=L(Ad(k^{-1})u^{T})f_{\sigma_{z}}(k) 
\end{equation*}
with $L$ the left action of $U(\mathfrak{g})$ on $C^{\infty}(G,S)$. Also%
\begin{equation*}
Ad(k^{-1})u^{T}=\sum_{i,j,I}a_{i,j,I}(k)n_{i}a_{j}Y^{I}
\end{equation*}
finite sum. Thus%
\begin{equation*}
d\pi_{\sigma_{z}}(u)f(k)=\sum a_{i,j,I}(k)d\sigma_{z}(n_{i}a_{j})\left(
L(Y^{I})f\right) (k) 
\end{equation*}%
\begin{equation*}
=\sum a_{i,j,I}(k)d\sigma_{z}(n_{i}a_{j})\left( Ad(k)^{-1}\left(
Y^{I}\right) ^{T}f\right) (k). 
\end{equation*}
Writing 
\begin{equation*}
Ad(k)^{-1}\left( Y^{I}\right) ^{T}=\sum_{\left\vert J\right\vert
\leq\left\vert I\right\vert }b_{J,I}(k)Y^{J}
\end{equation*}
we have%
\begin{equation*}
d\pi_{\sigma_{z}}(u)f(k)=\sum_{i,j,I,J}b_{J,I}(k)a_{i,j,I}(k)d%
\sigma_{z}(n_{i}a_{j})Y^{J}f(k). 
\end{equation*}
Since the sum is finite the all of the indices are bounded. Let $\omega$ be
a compact subset of $Z$ then for each fixed $J$%
\begin{equation*}
\sum_{i,j,I}|b_{J,I}(k)a_{i,j,I}(k)|\left\Vert
d\sigma_{z}(n_{i}a_{j})\right\Vert \leq C_{u,\omega,J}^{1},k\in
K,z\in\omega. 
\end{equation*}
Thus%
\begin{equation*}
\left\Vert d\pi_{\sigma_{z}}(u)f\right\Vert ^{2}\leq\sum
C_{u,\omega,L}^{1}C_{u,\omega,J}^{1}\left\langle Y^{L}f,Y^{J}f\right\rangle 
\end{equation*}%
\begin{equation*}
\leq C_{u,\omega}\left\Vert (1+C_{K})^{l}f\right\Vert ^{2}
\end{equation*}
with $l$ the maximum of the $\left\vert J\right\vert $ for the multi-indices
that appear in the formulas above.
\end{proof}

\section{Analytic families of $J$--modules}

Notation as in the previous section. Throughout this section analytic will
mean complex analytic in the context of a complex analytic manifold and real
analytic in the contest of a real analytic manifold.

\begin{proposition}
\label{J-family}Let $X$ be an analytic or complex manifold. Let $(\lambda,R)$
be a family of objects in $W(K,\mathbf{D})$ over $X$ and define $R_{x}\in
W(K,\mathbf{D})$ to be the module with action $\lambda(x,\cdot)$. Let 
\begin{equation*}
V=\mathcal{H}\otimes E\otimes R 
\end{equation*}
$(K$ act by the tensor product action with its action on $E$ trivial) and
let $T_{x}:V\rightarrow J(R_{x})$ be given by $T_{x}(h\otimes e\otimes
r)=\alpha_{x}(\mathrm{symm}(h)e)(1\otimes r)$\ with $\alpha_{x}$ the action
of $\ U(\mathfrak{g}_{\mathbb{C}})$ on $J(R_{x})$. If $%
\lambda(x,y)=T_{x}^{-1}\alpha_{x}(y)T_{x}$ then $(\lambda,V)$ is an analytic
family of objects in $\mathcal{H}(\mathfrak{g},K)$ based on $X$.
\end{proposition}

\begin{proof}
Let $\left\{ h_{i}\right\} $ be a basis of $\mathcal{H}$ such that for each $%
i$ there exists $\gamma\in\hat{K}$ such that $h_{i}\in\mathcal{H(\gamma)}$,
let $e_{j}$ be a basis of $E$, let $r_{m}$ be a basis of $R$ and let $%
Y_{1},...,Y_{n}$ be a basis of $\mathfrak{k}$. Then if $y\in U(\mathfrak{g}_{%
\mathbb{C}})$ 
\begin{equation*}
y\mathrm{symm}(h_{i})e_{j}z^{L}Y^{J}=%
\sum_{i_{1},j_{1},J_{1},L_{1}}b_{i_{1}j_{1}L_{1}J_{1},ijLK}(y)\mathrm{symm}%
(h_{i_{1}})e_{j_{1}}z^{L_{1}}Y^{J_{1}}. 
\end{equation*}
Thus%
\begin{equation*}
T_{x}^{-1}\alpha_{x}(y)T_{x}(h_{i}\otimes e_{j}\otimes r_{k})=\sum
b_{i_{1}j_{1}L_{1}J_{1},ij00}(y)h_{i_{1}}\otimes e_{j_{1}}\otimes\left(
\lambda_{x}(z^{L_{1}})Y^{J_{1}}r_{k}\right) . 
\end{equation*}
The proposition follows.
\end{proof}

Theorem \ref{decompositions} implies

\begin{lemma}
\label{covar}Let $R\in W(K,\mathbf{D})$ then 
\begin{equation*}
J(R)/\mathfrak{n}^{k+1}J(R)\cong\left( U(\mathfrak{n)/}\mathfrak{n}^{k+1}U(%
\mathfrak{n})\right) \otimes E\otimes H\otimes R_{|M}
\end{equation*}
as an $(\mathfrak{n},M)$-module with $\mathfrak{n}$ and $M$ acting trivially
on $E\otimes H$ and $\mathfrak{n}$ acting trivially on $R$.
\end{lemma}

Let $(\mu,R)$ be an analytic family of objects in $W(K,\mathbf{D)}$ over $X$%
. Let $R_{x},x\in X$ be the object in $W(K,\mathbf{D)}$ with $K$ acting by
its action on $R$ and $\mathbf{D}$ acting by $\mu_{x}=\mu(x,\cdot)$.

\begin{proposition}
\label{dependence} Let $p_{x,k}:J(R_{x})\rightarrow\left( U(\mathfrak{n)/}%
\mathfrak{n}^{k+1}U(\mathfrak{n})\right) \otimes E\otimes H\otimes R_{|M}$
be given by the projection of $J(R_{x})$ onto $J(R)/\mathfrak{n}^{k+1}J(R)$
composed with the isomorphism of $J(R)/\mathfrak{n}^{k+1}J(R)$ with $\left(
U(\mathfrak{n)/}\mathfrak{n}^{k+1}U(\mathfrak{n})\right) \otimes E\otimes
H\otimes R$. If $v\in J(R_{x})$ and $u\in U(\mathfrak{g}_{\mathbb{C}})$ then
the map 
\begin{equation*}
x\mapsto p_{x,k}(uv) 
\end{equation*}
is analytic from $X$ to $\left( U(\mathfrak{n)/}\mathfrak{n}^{k+1}U(%
\mathfrak{n})\right) \otimes E\otimes H\otimes R_{|M}$. In particular, if $%
p_{k}$ is the canonical projection of projection of $U(\mathfrak{n)}\otimes
E\otimes H\otimes R_{|M}$ onto $\left( U(\mathfrak{n)/}\mathfrak{n}^{k+1}U(%
\mathfrak{n})\right) \otimes E\otimes H\otimes R_{|M}$ then define $%
\sigma_{k,x}(q)p_{k}(v)=p_{x,k}(qv)$ for $q\in U(L(Q))$.
\end{proposition}

\begin{proof}
Let $x_{1},x_{2},...,x_{r}$ be a linearly independent set in $U(\mathfrak{n})
$ that projects to a basis in $U(\mathfrak{n)/}\mathfrak{n}^{k+1}U(\mathfrak{%
n})$ and let $x_{r+1},...$ be a basis of $\mathfrak{n}^{k+1}U(\mathfrak{n})$%
. Theorem \ref{decompositions} implies that if $Y_{1},...,Y_{n}$ is a basis
of $\mathfrak{k}$ and $h_{1},...,h_{r}$ is a basis for $\mathrm{symm}(E)%
\mathrm{symm}(H)$ then if $J$ is a multi-index of size $n$ and $I$ is a
multi-index of size then the set of elements%
\begin{equation*}
x_{l}z^{I}h_{m}Y^{J}
\end{equation*}
is a basis of $U(\mathfrak{g}_{\mathbb{C}})$( $z^{I}=z_{1}^{i_{1}}\cdots
z_{l}^{i_{l}}$ and the $z_{i}$ are the generators of $\mathbf{D}$). This
implies that if $u\in U(\mathfrak{g}_{\mathbb{C}})$ then 
\begin{equation*}
ux_{s}z^{I}h_{t}Y^{J}=%
\sum_{t_{1},I_{1.}t_{1},J_{1}}a_{s_{_{1}}J_{1}t_{1}L_{1},sItJ}(u)X^{J_{1}}z^{L_{1}}h_{i_{1}}Y^{J_{1}}. 
\end{equation*}
This implies that if we take a basis $v_{1},...,v_{d}$ of $R$ then the
elements $X^{J}h_{i}\otimes v_{j}$ form a basis of $J(V_{x})$. Thus if $u\in
U(\mathfrak{g}_{\mathbb{C}})$ then%
\begin{equation*}
ux_{s}h_{t}\otimes v_{j}=\sum
a_{s_{1}I_{1}t_{1}J_{1},s,0,t,0}(u)x_{s_{1}}z^{I_{1}}h_{t_{1}}Y^{J_{1}}%
\otimes v_{ju}= 
\end{equation*}%
\begin{equation*}
\sum
a_{s_{1}I_{1}t_{1}J_{1},s,0,t,0}(u)x_{s}h_{t}\otimes%
\mu_{x}(z^{I_{1}})Y^{J_{1}}v_{j}
\end{equation*}
Now apply $p_{k,x}$ getting the image of 
\begin{equation*}
\sum_{s_{1}\leq r}a_{s_{1}I_{1}t_{1}J_{1},s,0,t,0}(u)x_{s}h_{t}\otimes\left(
\mu_{x}(z^{I_{1}})Y^{J_{1}}v_{j}\right) 
\end{equation*}
The proposition follows from this formula.
\end{proof}

If $R\in W(K,\mathbf{D})$ the space $J(R)/\mathfrak{n}^{s+1}J(R)$ has a
natural structure of an $M$ module and an $\mathfrak{n}+\mathfrak{a}$
module. Since $\dim J(R)/\mathfrak{n}^{s+1}J(R)<\infty$ and $AN$ is a simply
connected Lie group $J(R)/\mathfrak{n}^{s+1}J(R)$ has a natural structure of
a finite dimensional continuous $Q$--module with action $\sigma_{s,R}$. Let
(as above) $p_{s}$ denote the natural surjection%
\begin{equation*}
p_{s}:J(R)\rightarrow J(R)/\mathfrak{n}^{s+1}J(R). 
\end{equation*}
If $k\in K$, $v\in J(R)$, define $S_{s,R}(v)(k)=p_{s,R}(kv)$, then $%
S_{s,R}(v)\in I(\sigma_{s,R}|_{M})$ and it is easily seen that $S_{s,R}\in%
\mathrm{Hom}_{\mathcal{H}(\mathfrak{g},K)}(J(R),(\pi_{\sigma_{s,R}},I(%
\sigma_{s,R}|_{M}))$. Combining the above results we have

\begin{theorem}
\label{J-morph}Let $(\mu,R)$ be an analytic (resp. continuous) family in $%
W(K,\mathbf{D})$ based on the manifold $X$. Let $(\lambda,V)$ be the
analytic family (as in Theorem \ref{J-family}) corresponding to $%
x\rightarrow J((\mu_{x},R))$. Then recalling that $V=\mathcal{H\otimes}%
E\otimes R$ define $T_{s}(x)(h\otimes e\otimes r)=S_{s},_{R_{x}}(\mathrm{symm%
}(h)e\otimes r).$ Then $T_{s}$ defines a homomorphism of the analytic family 
$(\lambda,V)$ to $(\xi,I(\sigma_{s,R_{x},}|_{M}))$ (in the sense of
Definition \ref{fam-hom}) \ with $\xi(x,y)=\pi_{\sigma_{s,R_{x}}}(y)$ and $%
\sigma_{s,R_{x}}$ is defined as above.
\end{theorem}

We will use the notation $J(R)$ for the analytic family associated with $%
x\rightarrow J((\mu_{x},R))$.

\section{Imbeddings of $J$--modules and their Hilbert family completions}

Let $X$ be a connected real or complex analytic manifold and let $(\mu,R)$
be an analytic family of objects in $W(K,\mathbf{D})$ based on $X.$ The
purpose of this section is to prove

\begin{theorem}
\label{Imbedding}Let the representation of $Q$, $\sigma_{k,x}$, on 
\begin{equation*}
W_{k}=\left( U(\mathfrak{n)/}\mathfrak{n}^{k+1}U(\mathfrak{n})\right)
\otimes E\otimes H\otimes R_{|M}
\end{equation*}
be as in Proposition \ref{dependence} and let $T_{k}(x)$ be the analytic
family as in Theorem \ref{J-morph}. If $\omega$ is a compact subset of $X$
then there exists $k_{\omega}$ such that if $x\in\omega$ then $T_{k}(x)$ is
injective for $k\geq k_{u}$.
\end{theorem}

\begin{proof}
This is a slight extension of a result in \cite{HOW}. Given $k$ then $%
(\sigma_{k,x},W_{k})$ as a composition series $W_{k,x}=W_{k,x}^{1}\supset
W_{k,x}^{2}\supset...\supset W_{k,x}^{r}\supset W_{k,x}^{r+1}=\{0\}$ and
each $W_{k,x}^{i}/W_{k,x}^{i+1}$ is isomorphic with the representation $%
(\lambda_{j,\nu_{j}},H_{\lambda_{j}})$ with $(\lambda_{j},H_{j})$ an
irreducible representation of $M$ and $\nu_{j}\in\mathfrak{a}_{\mathbb{C}%
}^{\ast}$ and $\lambda_{j,\nu}(man)=a^{\nu+\rho}\lambda_{j}(m)$ with $m\in
M,a\in A$ and $n\in N$. Also note that there is a natural $Q$--module exact
sequence%
\begin{equation*}
0\rightarrow\left( \mathfrak{n}^{k+1}U(\mathfrak{n)/n}^{k+2}U(\mathfrak{n}%
)\right) \otimes E\otimes H\otimes R_{|M}\rightarrow W_{k+1,x}\rightarrow
W_{k,x}\rightarrow0. 
\end{equation*}
We may assume that the composition series is consistent with this exact
sequence. This implies that the $\nu_{j}$ that appear in $W_{k}/W_{k+1}$ are
of the form $\mu+\alpha_{1}+...+\alpha_{k+1}$ with $\alpha_{i}$ a restricted
positive root (i.e. a weight of $\mathfrak{a}$ on $\mathfrak{n}$).

Now consider the corresponding exact sequence of $(\mathfrak{g},K)$--modules.%
\begin{equation*}
\overset{}{(\ast)}0\rightarrow I(\eta_{k,x})\rightarrow I(\sigma
_{k+1,x})\rightarrow I(\sigma_{k,x})\rightarrow0. 
\end{equation*}
The $(\mathfrak{g},K)$--modules $I(\sigma_{\nu})$ with $\sigma$ an
irreducible representation of $M$ with Harish-Chandra parameter $%
\Lambda_{\sigma}$ (for $Lie(M)_{\mathbb{C}}$) and $\nu\in\mathfrak{a}_{%
\mathbb{C}}^{\ast}$ have infinitesimal character with Harish-Chandra
parameter $\Lambda_{\sigma}+\nu$. We are finally ready to prove the theorem.

Let $C_{\omega}$ be the compact set $\cup_{x\in\omega}ch(J(R_{x}))$. Let $%
C_{\omega}=\cup_{i=1}^{k_{\omega}}(\Lambda_{i}+D_{i})$ with $D_{i}$ compact
in $\mathfrak{a}_{\mathbb{C}}^{\ast}$ and $k_{\omega}<\infty$. \ Assume that
the result is false for $\omega$. Then for each $j$ there exists $k\geq j$
and $x$ such that $\ker T_{k}(x)\neq0$ but $\ker T_{k+1}(x)=0$. Label the
Harish -Chandra parameters that appear in $I(\sigma_{o,x})$, $%
\Lambda_{1}+\nu _{1},...,\Lambda_{s}+\nu_{s}$ with $\Lambda_{i}\in
Lie(T)^{\ast}$ and $\nu _{i}\in\mathfrak{a}_{\mathbb{C}}^{\ast}$ (recall
that we have fixed a maximal torus of $M$). The above observations imply
that $ch(J(R_{x}))$ contains an element of the form $\Lambda+\nu_{i_{k}}+%
\beta_{k}$ with $\beta_{k}$ a sum of $k$ positive roots, $\Lambda\in
Lie(T)^{\ast}$ and $1\leq i_{k}\leq s$. We now have our contradiction $%
\nu_{i_{k}}+\beta_{k}\in\cup D_{i}$ which is compact. But the set of $%
\nu_{i_{k}}+\beta_{k}$ is unbounded.
\end{proof}

\begin{theorem}
\label{globalize-J}Let $U\subset Z$ be open with compact closure. There
exists a continuous family $(\pi,H)$ of Hilbert representations of $G$ ( see
Definition \ref{Hilbert-Fam}) based on $U$ such that the continuous family
of $(\mathfrak{g},K)$--modules $(d\pi,H_{K}^{\infty})$ is isomorphic with
the analytic family $z\mapsto J(L_{z})$ of objects in $\mathcal{H}(\mathfrak{%
g},K)$ based on on $U$ (thought of as a continuous family). Furthermore, the
family $(\pi,H)$ is smoothable (see Definition \ref{accessible}) .
\end{theorem}

\begin{proof}
Let $\gamma\in\hat{K}$ then Theorem \ref{decompositions} 2.implies 
\begin{equation*}
\dim J(L_{z})(\gamma)=\dim E\dim\gamma\dim\mathrm{Hom}_{K}(V_{\gamma },%
\mathcal{H}\otimes L). 
\end{equation*}
for every $z\in Z$. In particular it is independent of $z$. Theorem \ref%
{Imbedding} implies that there exists $k$ and for each $u\in U$ the map 
\begin{equation*}
T_{k,L_{u}}:J(L_{u})\rightarrow I(\sigma_{k,L_{u}}) 
\end{equation*}
is injective. Note that the space of $K$--finite vectors in $I(\sigma
_{k,L_{u}})$ is the $K$--finite induced representation $Ind_{M}^{K}(%
\sigma_{k,L_{|M}})$ and hence independent of $u$. Let \thinspace $%
(H_{1},\left\langle ...,...\right\rangle )$ be the Hilbert space completion
of $Ind_{M}^{K}(\sigma_{k,L_{|M}})$ corresponding to unitary induction from $%
M$ to $K$. This gives an smoothable analytic family of Hilbert
representations of $G$ (Proposition \ref{Par-access} , $\mu_{z}$.
Proposition \ref{Hilb-access} now implies the result.
\end{proof}

\section{The main theorem}

\begin{theorem}
\label{hilb-main}Let $(\pi,V)$ be an analytic family of objects in $\mathcal{%
H}(\mathfrak{g},K)$ based on the analytic manifold $X$. Let $\ x_{o}\in X$
then there exists, $U$, an open neighborhood of $x_{o}$ in $X,$ and a
smoothable, continuous family of Hilbert representations $(\mu_{U},H_{U})$
such that the family $(d\mu_{U},\left( H_{U}\right) _{K}^{\infty})$ is
isomorphic with $(\pi_{|U},V)$(as a continuous family).
\end{theorem}

\begin{proof}
Let $U_{1}$ be an open neighborhood of $x_{o}$ in $X$ with compact closure.
Then Theorem \ref{loc-finite} implies that there exists $F_{U_{1}}^{0}\subset%
\hat{K}$ a finite subset such that $\pi_{x}(U(\mathfrak{g}_{\mathbb{C}%
}))\sum_{\gamma\in F_{U}^{0}}V(\gamma)=V$. Let $R^{0}=\sum_{\gamma\in
F_{U}}V(\gamma)$. $R^{0}$ is invariant under the action $\pi_{x}(\mathbf{D)}$
for all $x\in X$. This implies that $(\left( \pi_{|U}\right) |_{\mathbf{D}%
},R^{0})$ defines an analytic family of objects in $W(K,\mathbf{D})$ based
on $U_{1}$. Let $J(R^{0})$ be the corresponding $J$--family. Then we have
the surjective analytic homomorphism of families%
\begin{equation*}
\begin{array}{ccccc}
& T_{0} &  &  &  \\ 
J(R^{0}) & \rightarrow & V_{|U} & \rightarrow & 0%
\end{array}
\end{equation*}
with $T_{0}(x)$ mapping $J(R_{x}^{0})$ onto $V$ for all $x\in U_{1}$. Let $%
(\sigma,(H^{0},(...,...)))$ be the smoothable, continuous family of Hilbert
representations based on $U_{1}$ corresponding to $J(R^{0})$ as in Theorem %
\ref{globalize-J}. Let $U$ be an open neighborhood of $x_{o}$ contained in $%
U_{1}$ such that $\overline{U}$ is contractible. The theorem now follows
from Proposition \ref{hilb-surj}.
\end{proof}

The main result is

\begin{theorem}
\label{main}Let $(\pi,V)$ be a holomorphic family of objects in $\mathcal{H}(%
\mathfrak{g},K)$ based on the connected complex manifold $X$. Then there
exists a holomorphic family of smooth Fr\'{e}chet representations, $%
(\lambda,W)$ based on $X$ that globalizes $(\pi,V)$.
\end{theorem}

\begin{proof}
The above theorem implies that there is an open covering, $\{U_{\alpha}\}$,
of $X$ and for each $\alpha$ a continuous family of smoothable, admissible
Hilbert representations based on $U_{\alpha}$, $(\sigma_{\alpha},H_{\alpha})$%
, such that $((d\sigma_{\alpha})_{x},\left( H_{\alpha}\right) _{K})=(\pi
_{x},V),x\in U_{\alpha}$. Proposition \ref{Cont-diff} combined with Theorem %
\ref{cont-hol} implies that $(\sigma_{\alpha},H_{\alpha}^{\infty})$ (here $%
H_{\alpha}^{\infty}$ is the space of $C^{\infty}$ vectors with respect to $K$%
). is a Holomorphic family of smooth Fr\'{e}chet representations of $G$
based on $U_{\alpha}$. The isomorphism of categories implies that if $x\in
U_{\alpha}\cap U_{\beta}$ then $H_{\alpha}^{\infty}=H_{\beta}^{\infty}$ and $%
\left( \sigma_{\alpha}\right) _{x}(g)|_{H_{\alpha}^{\infty}}=\left(
\sigma_{\beta}(g)\right) _{x}|_{H_{\beta}^{\infty}}$ for all $g\in G$. Thus
we can define $W=H_{\alpha}^{\infty}$ the common value for all $\alpha$
(since $X$ is connected) \ and if $x\in X$ then $\sigma_{x}(g)=\left( \sigma
_{\alpha}\right) _{x}(g)$ for $\alpha$ such that $x\in U_{\alpha}$.
\end{proof}

This can be interpreted in the following way:

\begin{corollary}
\label{main2}Let $T$ be the inverse functor to the $K$--finite functor $%
\mathcal{HF}(G)\rightarrow\mathcal{H}(\mathfrak{g},K)$ and let $(\pi,V)$ is
an holomorphic family of objects in $\mathcal{H}(\mathfrak{g},K)$ over the
connected complex manifold $X$. If $T((\pi_{x},V))=(\lambda_{x},\overline {%
V_{x}})$ then

1. For all $x,y\in X,\overline{V_{x}}=\overline{V_{y}}$ as subspaces of $%
\prod_{\gamma\in\hat{K}}V(\gamma)$ and as Fr\'{e}chet spaces. Set $\overline{%
V}$ equal to the common value.

2. The map $x,g,v\longmapsto\lambda_{x}(g)v$ is continuous from $X\times
G\times\overline{V}$ to $\bar{V}$, linear in $v$ and $C^{\infty}$ in $g$and
holomorphic in $x$.
\end{corollary}

\section*{Appendices}

\appendix

\section{$G$--$C^{\infty}$ vectors and $K$--$C^{\infty}$ vectors}

Let $G$ be a real reductive group with a fixed maximal compact subgroup $K$
and let $\theta$ be the corresponding Cartan involution. Fix a symmetric
bilinear form $B$ on $Lie(G)$ such that $\left\langle X,Y\right\rangle
=-B(\theta X,Y)$ is positive definite. Let $C$ and $C_{K}$ be the Casimir
operators of $G$ and $K$ respectively corresponding to $B$ and we set $%
\Delta=C-2C_{K}$. We observe that $\Delta=\sum X_{i}^{2}$ for $%
X_{1},...,X_{m}$ an orthonormal basis of $Lie(G)$ relative to $\left\langle
...,...\right\rangle $. As a left invariant operator on $G$, $\Delta$ is an
elliptic and invariant under $K$. Let $(\pi,H)$ be a Hilbert representation
of $G$ and set $V=\left( H^{\infty}\right) _{K}$ . Let $Z$ be the completion
of $V$ relative to the seminorms $q_{l}(v)=\left\Vert \Delta^{l}v\right\Vert
,l=0,1,2,...$ Then since $q_{0}=\left\Vert ...\right\Vert $, $Z$ can be
looked upon as a subspace of $H$. Also $H^{\infty}$ is the completion of $V$
using the seminorms $s_{x}(v)=\left\Vert xv\right\Vert $ with $x\in U(%
\mathfrak{g})$. Thus $Z\supset H^{\infty}$.

\begin{lemma}
\label{roe}$Z=H^{\infty}$. Furthermore, the topology on $H^{\infty}$ is
given by the semi-norms $q_{l}$. That is $H^{\infty}=H^{\infty_{K}}.$
\end{lemma}

\begin{proof}
We note that the second assertion is a direct consequence of the closed
graph theorem (c.f. \cite{Treves}) and the first assertion. We will now
prove the first assertion. Let $v\in Z\subset H$. We must prove that $v\in
H^{\infty}$. Let $v_{j}\in V$ be a sequence converging to $v$ in the
topology of $Z$. Let $w\in H$ then for all $j$%
\begin{equation*}
\Delta^{k}(\pi(g)v_{j},w)=(\pi(g)\Delta^{k}v_{j},w). 
\end{equation*}
Set $p_{k}(v)=\sum_{j=0}^{k}q_{k}(v)$. Noting that $p_{l}(%
\Delta^{k}v)+p_{k-1}(v)=p_{l+k}(v)$ and $p_{l}(v)\leq p_{l+1}(v)$, we see
that for fixed $k$ the sequence $\{\Delta^{k}v_{j}\}_{j}$ converges to $u_{k}
$ in $Z$.

We assert that the function $g\mapsto(\pi(g)v,w)$ is $C^{\infty}$. Since $%
w\in H$ is arbitrary this would imply that the map $g\mapsto\pi(g)v$ is
weakly $C^{\infty}$. But a weakly $C^{\infty}$ map of a finite dimensional
manifold into a Hilbert space is strongly $C^{\infty}$(c.f. \cite{Groth-sm}%
). This is exactly the statement that $v$ is a $C^{\infty}$ vector. We now
prove the assertion. We first show that if we look upon the continuous
function $h(g)=(\pi(g)v,w)$ as a distribution on $G$ (using the Haar measure
on $G)$ then in the distribution sense 
\begin{equation*}
\Delta^{k}h(g)=(\pi(g)u_{k},w). 
\end{equation*}
Indeed, let $f\in C_{c}^{\infty}(G)$ then%
\begin{equation*}
\int_{G}h(g)\Delta^{k}f(g)dg=\lim_{j\rightarrow\infty}\int_{G}(\pi
(g)v_{j},w)\Delta^{k}f(g)dg= 
\end{equation*}%
\begin{equation*}
\lim_{j\rightarrow\infty}\int_{G}\Delta^{k}(\pi(g)v_{j},w)f(g)dg=\lim
_{j\rightarrow\infty}\int_{G}(\pi(g)\Delta^{k}v_{j},w)f(g)dg= 
\end{equation*}%
\begin{equation*}
\int_{G}(\pi(g)u_{k},w)f(g)dg 
\end{equation*}
as asserted. Since $\Delta$ is elliptic, local Sobelev theory (c.f. [F,
Chapter 6]) implies that $h\in C^{\infty}(G)$.
\end{proof}

\begin{proposition}
\label{C-K}If $(\pi,H)$ is an admissible Hilbert representation of $G$ such
that there exists a polynomial%
\begin{equation*}
f(x)=x^{m}-\sum_{j=0}^{m-1}c_{j}x^{j}
\end{equation*}
such that $f(C)=0$ on $H^{\infty}$ then the topology of $H^{\infty}$ is
given by the semi-norms $p_{l}(v)=\left\Vert (I+C_{K})^{l}v\right\Vert
,l=0,1,2,...$That is, the $K-C^{\infty}$ vectors, $H^{\infty_{K}}$ are the
same as the $G-C^{\infty}$ vectors, $H^{\infty}$.
\end{proposition}

\begin{proof}
This result will be proved by induction on $m$. If $m=1$ then $C$ acts by $%
c=c_{0}$ on $H^{\infty}$. Note that $\Delta=C-2C_{K}$. So if $v\in H^{\infty
}$%
\begin{equation*}
\left\Vert \Delta^{k}v\right\Vert =\left\Vert \sum_{j=0}^{k}(-2)^{j}\binom {k%
}{j}C^{k-j}C_{K}^{j}v\right\Vert _{{}}\leq 
\end{equation*}%
\begin{equation*}
\sum_{j=0}^{k}(2)^{j}\binom{k}{j}|c|^{\substack{  \\ k-j}}\left\Vert
C_{K}^{j}v\right\Vert \leq\sum_{j=0}^{k}(2)^{j}\binom{k}{j}%
|c|^{k-j}\left\Vert (1+C_{K})^{j}v\right\Vert . 
\end{equation*}
Now assume the result if the degree is $m-1\geq1$. Let $H^{\omega}$ denote
the space of analytic vectors in $H^{\infty}$. then the $K$--finite vectors $%
V=H_{K}$ are contained in $H^{\omega}$ \ If $c\in\mathbb{C}$ and $V_{c}$ $%
=\{v\in V|Cv=cv\}\neq0$ then let $H_{1}$ be the Hilbert space completion of $%
V_{c}$ then $H_{1}$ is $G$--invariant and $\left( H_{1}\right) _{K}=V_{c}$
and $\left( H/H_{1}\right) _{K}=H_{K}/V_{c}$. The first part of the proof
implies that the seminorms $p_{l}$ define the topology on $H_{1}^{\infty}$.
The correspondence $U\rightarrow U^{\infty}$ is an exact functor from the
category of strongly continuous representations of $G$ to the category of
smooth Fr\'{e}chet modules (\cite{Warner1}, Proposition 4.4.1.11 p. 260).
Setting 
\begin{equation*}
g(x)=\frac{f(x)}{x-c}
\end{equation*}
we have $g(C)$ is zero on $\left( H/H_{1}\right) ^{\infty}$ (since it is $0$
on $V/V_{c}$). Thus we have the commutative diagram%
\begin{equation*}
\begin{array}{ccccccccc}
0 & \rightarrow & H_{1}^{\infty} & \rightarrow & H^{\infty} & \rightarrow & 
(H/H_{1})^{\infty} & \rightarrow & 0 \\ 
&  & \downarrow &  & \downarrow &  & \downarrow &  &  \\ 
0 & \rightarrow & H_{1}^{\infty_{K}} & \rightarrow & H^{\infty_{K}} & 
\rightarrow & \left( H/H_{1}\right) ^{\infty_{K}} & \rightarrow & 0%
\end{array}
\end{equation*}
with the right-most and left-most vertical arrows isomorphisms. This implies
that the middle vertical arrow is also an isomorphism completing the
induction.
\end{proof}

\begin{corollary}
If $(\pi,H)$ is an admissible finitely generated Hilbert representation of $G
$ then the $K$--$C^{\infty}$ vectors are the same as the $G$--$C^{\infty}$
vectors.
\end{corollary}

\begin{proof}
There exists a finite subset $F\subset\hat{K}$ such that $U(\mathfrak{g}%
)\sum_{\gamma\in F}H(\gamma)$. Clearly, $C\sum_{\gamma\in F}H(\gamma
)\subset\sum_{\gamma\in F}H(\gamma)$. Let $p(x)$ be the minimal polynomial
of $C_{|\sum_{\gamma\in F}H(\gamma)}$. Then $p(C)=0$ on $H_{K}$ and hence on 
$H^{\infty}$.
\end{proof}

\section{Hilbert families}

Let $G$ be a locally compact topological group and let $X$ be a locally
compact metric space. All Hilbert spaces in this appendix (indeed, in this
paper) will be separable.

\begin{definition}
\label{Hilbert-Fam}A \ continuous family of Hilbert representations of $G$
over $X$ is a pair $(\sigma,H)$ with $H$ a Hilbert space and $\sigma:X\times
G\rightarrow GL(H)$ (the continuos invertible operators on $H$ with the
strong operator topology) a continuous map such that $\sigma_{x}(g)=%
\sigma(x,g)$ defines a representation of $G$ for every $x\in G$.
\end{definition}

The following lemma is Lemma 1.1.3 in \cite{RRG1-11} taking into account
dependence on parameters. The proof is essentially the same using the local
compactness of $X.$

\begin{lemma}
Let $X$ be a locally compact metric space and let $H$ be a Hilbert space.
Assume that for each $x\in X$, $\pi_{x}:G\rightarrow GL(H)$ (bounded
invertible operators such that

1) If $\omega\subset X$ and $\Omega\subset G$ are compact subsets of $X$ and
of $G$ respectively then there exists a constant $C_{\omega,\Omega}$ such
that $\left\Vert \pi_{x}(g)\right\Vert \leq C_{\omega,\Omega}$ \ for $x\in
\omega,g\in\Omega.$

2) The map $x,g\rightarrow\left\langle \pi_{x}(g)v,w\right\rangle $ is
continuous for all $v,w\in H$.

Then $(\pi,H)$ is a continuous family of representations of $G$ based on $X$
and conversely if $(\pi,H)$ is a continuous family of Hilbert
representations then 1) and 2) \ are satisfied.
\end{lemma}

An immediate corollary is

\begin{corollary}
\label{Conj-dual}Let $(\pi,H)$ be an admissible, continuous family of
Hilbert representations of $G$ based on the locally compact metric space $X$%
. Set for each $x\in X$, $\hat{\pi}_{x}(g)=\pi_{x}(g^{-1})^{\ast}$ then $(%
\hat{\pi},H)$ is a continuous, admissible family of Hilbert representations
of of $G$ based on $X$.
\end{corollary}

\section{\label{norms}Norms}

Let $\left\Vert g\right\Vert $ be a norm on $G,$ that is a continuous
function from \thinspace$G$ to $\mathbb{R}_{>0}$ (the positive real numbers)
such that

1. $\left\Vert k_{1}gk_{2}\right\Vert =\left\Vert g\right\Vert
,k_{1},k_{2}\in K,g\in G$,

2. $\left\Vert xy\right\Vert \leq\left\Vert x\right\Vert \left\Vert
y\right\Vert ,x,y\in G$,

3. The sets $\left\Vert g\right\Vert \leq r<\infty$ are compact.

4. If $X\in\mathfrak{p}$ then if $t\geq0$ then $\log\left\Vert \exp
tX\right\Vert =t\log\left\Vert \exp X\right\Vert .$

If $(\sigma,V)$ is a finite dimensional representation of $G$ with compact
kernel. Assume also that $\left\langle ...,...\right\rangle $ is an inner
product on $V$ that is $K$--invariant and is such that the elements $%
\sigma(\exp(X))$ are self adjoint for $X\in\mathfrak{p}_{o}$. If $\left\Vert
\sigma(g)\right\Vert $is the operator norm of $\sigma(g)$ then $\left\Vert
g\right\Vert =\left\Vert \sigma(g)\right\Vert $ is a norm on $G$. Taking the
representation on $V\oplus V$ given by%
\begin{equation*}
\left[ 
\begin{array}{cc}
\sigma(g) &  \\ 
& \sigma(g^{-1})^{\ast}%
\end{array}
\right] 
\end{equation*}
then we may (and will) assume in addition

5. $\left\Vert g\right\Vert =\left\Vert g^{-1}\right\Vert $.

Note that 5. implies that $\left\Vert g\right\Vert \geq1$.

Using the same proof as Lemma 2.A.2.1 in \cite{RRG1-11}(which we give for
the sake of completeness) one can prove

\begin{lemma}
\label{loc-umod-growth}If $(\pi,H)$ is a continuous family of Hilbert
representations over $X$ and if $\omega$ is a compact subset of $X$ then
there exist constants $C_{\omega},r_{\omega}$ such that%
\begin{equation*}
\left\Vert \pi_{x}(g)\right\Vert \leq C_{\omega}\left\Vert g\right\Vert
^{r_{\omega}}. 
\end{equation*}
\end{lemma}

\begin{proof}
Let 
\begin{equation*}
B_{1}=\{g\in G|\left\Vert g\right\Vert \leq1\}. 
\end{equation*}
If $v\in H$ and $(x,g)\in\omega\times B_{1}$ then $\sup\left\Vert \pi
_{x}(g)v\right\Vert <\infty$ by strong continuity. The principle of uniform
boundedness (c.f. \cite{ReedSimon},III.9,p.81)\ implies that there exists a
constant, $R$, such that $\left\Vert \pi_{x}(g)\right\Vert \leq R$ for $%
(x,g)\in\omega\times B_{1}$. Let $r=\log R$. In particular if $k\in K$ then 
\begin{equation*}
\left\Vert \pi_{x}(kg)\right\Vert \leq\left\Vert \pi_{x}(k)\right\Vert
\left\Vert \pi_{x}(g)\right\Vert \leq R\left\Vert \pi_{x}(g)\right\Vert . 
\end{equation*}
Also, 
\begin{equation*}
\left\Vert \pi_{x}(g)\right\Vert =\left\Vert
\pi_{x}(k^{-1})\pi_{x}(kg)\right\Vert \leq R\left\Vert
\pi_{x}(kg)\right\Vert . 
\end{equation*}
Thus for all $k\in K,g\in G$%
\begin{equation*}
R^{-1}\left\Vert \pi_{x}(g)\right\Vert \leq\left\Vert \pi_{x}(kg)\right\Vert
\leq R\left\Vert \pi_{x}(g)\right\Vert . 
\end{equation*}
Let $X\in\mathfrak{p}$, $X\neq0$ and let $j$ be such that 
\begin{equation*}
j<\log\left\Vert \exp X\right\Vert \leq j+1 
\end{equation*}
then%
\begin{equation*}
\log\left\Vert \pi_{x}(\exp X)\right\Vert \leq\log\left\Vert \pi_{x}(\exp(%
\frac{X}{j+1})\right\Vert ^{j+1}\leq r(j+1)\leq r+r\log\left\Vert \exp
X\right\Vert . 
\end{equation*}
Thus%
\begin{equation*}
\left\Vert \pi_{x}\left( \exp X\right) \right\Vert \leq R\left\Vert \exp
X\right\Vert ^{r},X\in\mathfrak{p.}
\end{equation*}
If $g\in G$ then $g=k\exp X$ with $k\in K$ and $X\in\mathfrak{p}$ so%
\begin{equation*}
\left\Vert \pi_{x}(g)\right\Vert =\left\Vert \pi_{x}(k\exp X)\right\Vert
\leq R^{2}\left\Vert \exp X\right\Vert ^{r}=R^{2}\left\Vert g\right\Vert
^{r}. 
\end{equation*}
Take $C_{\omega}=R^{2}$ and $r_{\omega}=r$.
\end{proof}

\section{\label{an-gK}Continuous and analytic families of $(\mathfrak{g},K)$
modules.}

Let $G$ be a reductive group with fixed maximal compact subgroup, $K$. In
this section $X$ will denote a connected, paracompact real analytic or
complex manifold.

\begin{definition}
If $V$ is a vector space over $\mathbb{C}$ then a continuous,real analytic
or holomorphic function from $X$ to $V$ is a map $f:X\rightarrow V$ such
that for each $x\in X$ there exists, $U$, an open neighborhood of $x$ in $X$
such that the following two conditions are satisfied:

1. $\dim\mathrm{span}_{\mathbb{C}}\{f(x)|x\in U\}<\infty$.

2. $f:U\rightarrow\mathrm{span}_{\mathbb{C}}\{f(x)|x\in U\}$ is respectively
continuous, real analytic or continuous.
\end{definition}

\begin{definition}
A holomorphic,analytic or continuous family of admissible $(\mathfrak{g},K)$%
--modules over $X$ is a pair, $(\mu,V)$, of an admissible $(\mathfrak{k},K)$%
--module, $V$, and 
\begin{equation*}
\mu:X\times U(\mathfrak{g})\rightarrow\mathrm{End}(V) 
\end{equation*}
such that $x\mapsto\mu(x,y)v$ is respectively holomorphic, analytic or
continuous for all $y\in U(\mathfrak{g})$, $v\in V$ and if we set $\mu
_{x}(y)=\mu(x,y)$ for $y\in U(\mathfrak{g})$ then $(\mu_{x},V)$ is an
admissible $(\mathfrak{g},K)$--module. It will be called a family of objects
in $\mathcal{H}(\mathfrak{g},K)$ if each $(\mu_{x},V)$ is finitely generated.
\end{definition}

\begin{definition}
\label{fam-hom}If $(\lambda,V)$ and $(\mu,W)$ are analytic or continuous
families of objects in $\mathcal{H}(\mathfrak{g},K)$ over $X$ then a
homomorphism of the family $(\lambda,V)$ to $(\mu,W)$ is a map 
\begin{equation*}
T:X\rightarrow\mathrm{Hom}_{\mathbb{C}}(V,W) 
\end{equation*}
such that

1. $x\mapsto T(x)v$ is an analytic or continuous map of $X$ to $W$ for all $%
v\in V.$

2. $T(x)\in\mathrm{Hom}_{\mathcal{H}(\mathfrak{g},K)}(V_{x},W_{x})$ with $%
V_{x}=(\lambda_{x},V),W_{x}=(\mu_{x},W)$.
\end{definition}

\begin{lemma}
Let $(\pi,H)$ be a continuous family of admissible Hilbert representations
of $G$ over $X$ and denote by $d\pi_{x}$the action of $\mathfrak{g}$ on $%
H_{K}^{\infty}$ (the $K$--finite $C^{\infty}$--vectors). Then $(d\pi,H_{K})$
is a continuous family of admissible $(\mathfrak{g},K)$--modules based on $X$%
.
\end{lemma}

\begin{proof}
If $\gamma\in\hat{K}$ then $C_{c}^{\infty}(\gamma;G)$ denotes the space of
all $f\in C_{c}^{\infty}(G)$ such that 
\begin{equation*}
\int_{K}\chi_{\gamma}(k)f(k^{-1}g)dk=f(g),g\in G 
\end{equation*}
with $\chi_{\gamma}$ the character of $\gamma$. Then 
\begin{equation*}
H(\gamma)=\pi_{x}(C_{c}^{\infty}(\gamma;G))H. 
\end{equation*}
We also note that if $Y\in\mathfrak{g},f\in C_{c}^{\infty}(\gamma;G)$ and $%
v\in H$ then 
\begin{equation*}
d\pi_{x}(Y)\pi_{x}(f)v=\pi_{x}(Yf)v 
\end{equation*}
with $Yf$ the usual action of $Y\in\mathfrak{g}$ on $C^{\infty}(G)$ as a
left invariant vector field. Thus, if $v\in H_{K}$ and $y\in U(\mathfrak{g}_{%
\mathbb{C}})$ then the map%
\begin{equation*}
x\longmapsto d\pi_{x}(y)v 
\end{equation*}
is continuous.
\end{proof}

\section{\label{VanderN}Some results of Vincent van der Noort}

Throughout this section $Z$ will denote a connected real or complex analytic
manifold. We will use the terminology analytic to mean complex analytic or
real analytic depending on the context.

We continue the notation of the previous sections. In particular $G$ is a
real reductive group of inner type.

We denote (as is usual) the standard filtration of $U(\mathfrak{g)}$, by 
\begin{equation*}
...\subset U^{j}(\mathfrak{g)}\subset U^{j+1}(\mathfrak{g)\subset...}
\end{equation*}
Let $V$ be an admissible $(Lie(K),K)$ module. We note that if $E\subset V$
is a finite dimensional $K$--invariant subspace of $V$ then there exists a
finite subset $F_{j,E}\subset\hat{K}$ such that 
\begin{equation*}
U^{j}(\mathfrak{g})\otimes E\cong\sum_{\gamma\in
F_{j,E}}m_{\gamma,j}V_{\gamma}. 
\end{equation*}
If $v\in V$ we denote by $E_{v}$ the span of $Kv$ in $V$.

The purpose of this section is to prove a theorem of van der Noort which
first appeared in his thesis \cite{VanderNoort}. We include the details only
because he is not expected to publish it. In his thesis he studied the
holomorphic case. Our exposition follows his original line.

Fix a maximal torus, $T$, of $M$ then $\mathfrak{h}_{o}=Lie(T)\oplus 
\mathfrak{a}$ is a Cartan subalgebra of $\mathfrak{g}_{o}$. As usual, set $%
\mathfrak{h}$ equal to the complexification of $\mathfrak{h}_{o}$. We
parametrize the homomorphisms of $Z(\mathfrak{g})$ to $\mathbb{C}$ by $%
\chi_{\Lambda}$ for $\Lambda\in\mathfrak{h}^{\ast}$using the Harish--Chandra
parametrization. \ Since $M$ is compact we endow $\hat{M}$ with the discrete
topology. Note that if $C$ is a compact subset (sorry of the over use of $C$
the Casimir operator will not appear in this section) of $\hat{M}\times%
\mathfrak{a}_{\mathbb{C}}^{\ast}$ then there exist a finite number of
elements $\xi_{1},...,\xi_{r}\in\hat{M}$ and compact subsets, $D_{j}$, of $%
\mathfrak{a}_{\mathbb{C}}^{\ast}$ such that%
\begin{equation*}
C=\cup_{j=1}^{r}\xi_{j}\times D_{j}\text{.}
\end{equation*}

If $\xi\in\hat{M}$ and $\nu\in\mathfrak{a}_{\mathbb{C}}^{\ast}$ then set $%
\sigma_{\xi,\nu}(man)=\xi(m)a^{\nu+\rho}$ ($\rho(H)=\frac{1}{2}%
tr(adH_{|Lie(N)})$)$,$ $H\in\mathfrak{a}$), $a^{\nu}=\exp(\nu(H))$ $a=\exp
(H)$, $\xi$ is taken to be a representative of the class $\xi$. $H^{\xi,\nu}$
is $I(\sigma_{\xi.\nu})$ which equals as a $K$--module $H^{\xi}=Ind_{M}^{K}(%
\xi$\thinspace$)$. If $f\in H^{\xi}$ set $f_{\nu}(nak)=a^{\nu+\rho
}f(k),n\in N,a\in A,k\in K$. $A_{\bar{P}}(\nu)$ is the corresponding
Kunze-Stein intertwining operator (c.f. \cite{HarHomSP}, 8.10.18. p.241).

\begin{proposition}
\label{indfinite}Let $\xi\in\widehat{M}$ and let $\Omega\subset\mathfrak{a}_{%
\mathbb{C}}^{\ast}$ be compact. There exists $F\subset\widehat{K}$ such that 
$\pi_{\xi,\nu}(U(\mathfrak{g}))\left( \sum_{\gamma\in
F}H^{\xi}(\gamma)\right) =H^{\xi}$ for all $\nu\in\Omega$.
\end{proposition}

The proof of this result will use the following

\begin{lemma}
If $\nu_{o}\in\mathfrak{a}_{\mathbb{C}}^{\ast}$ then there exists an open
neighborhood of $\nu_{o}$, $U_{\nu_{o}}$, and a finite subset $F=F_{\nu_{o}}$
of $\widehat{K}$ such that $\pi_{\xi,\nu}(U(\mathfrak{g}))\left( \sum
_{\gamma\in F}H^{\xi}(\gamma)\right) =H^{\xi}$ for all $\nu\in U_{\nu_{o}}$.
\end{lemma}

\begin{proof}
If $\gamma\in\hat{K}$ fix $W_{\gamma}\in\gamma$. If $\func{Re}(\nu,\alpha)>0$
for all $\alpha\in\Phi^{+}$ and if $\gamma\in\widehat{K}$ and $A_{\overline{P%
}}(\nu)H^{\xi}(\gamma)\neq0$ then $\pi_{\xi,\nu}(U(\mathfrak{g}))\left(
H^{\xi}(\gamma)\right) =H^{\xi}$(c.f. \cite{RRG1-11}, Theorem 5.4.1 (1)).
Fix such a $\gamma_{\nu}$ (which always exists since the operator $A_{%
\overline{P}}(\nu)\neq0$), take $F_{\nu}=\{\gamma_{\nu}\}$ and $U_{\nu}$ an
open neighborhood of $\nu$ such that $A_{\overline{P}}(\mu)H^{\xi}(\gamma_{%
\nu})\neq0$ for $\mu\in U$. Let $\nu\in\mathfrak{a}_{\mathbb{C}}^{\ast}$ be
arbitrary. There exists a positive integer, $k$, such that $\func{Re}%
(\nu+k\rho,\alpha)>0$ for all $\alpha\in\Phi^{+}$ and such that $k\rho$ is
the highest weight of a finite dimensional spherical representation, $%
V^{k\rho},$ of $G$ relative to $\mathfrak{a}$. The lowest weight of $%
V^{k\rho}$ relative to $\mathfrak{a}$ is $-k\rho$ and $M$ acts trivially on
that weight space thus $H_{K}^{\xi,\nu+k\rho}\otimes V^{k\rho}$ has $%
H_{K}^{\xi,\nu}$ as a quotient representation (see \cite{HarHomSP}%
,8.5.14,15). Take $F_{\nu}$ to be the set of $K$--types that occur in both $%
W_{\gamma_{\nu+k\rho}}\otimes V^{k\rho}$ and $H^{\xi}$ and $%
U_{\nu}=U_{\nu+k\rho}-k\rho$.
\end{proof}

We now prove the proposition. By the lemma above for each $\nu\in\Omega$
there exists $F_{\nu}$ and $U_{\nu}$ as in the statement of the lemma. The $%
U_{\nu}$ form an open covering of $\Omega$ which is assumed to be compact.
Thus there exist a finite number $\nu_{1},...,\nu_{r}\in\Omega$ such that 
\begin{equation*}
\Omega\subset\cup_{i=1}^{r}U_{\nu_{i}}\text{.}
\end{equation*}
Take $F=\cup_{i-1}^{r}F_{\nu_{i}}.$ This proves the proposition.

\begin{lemma}
Let $\chi_{\xi},_{\nu}$ denote the infinitesimal character of $\pi_{\xi,\nu}$%
. If $C$ is a compact subset of $\mathfrak{h}^{\ast}$ then 
\begin{equation*}
\{(\xi,\nu)\in\{\hat{M}\times\mathfrak{a}_{\mathbb{C}}^{\ast}|\chi_{\xi},_{%
\nu}=\chi_{\Lambda},\Lambda\in C\} 
\end{equation*}
is compact.
\end{lemma}

\begin{proof}
Fix a system of positive roots for $(M^{0},T)$ ($M^{0}$ the identity
component of $M$). If $\lambda_{\xi}$ is the highest weight of $\xi$
relative to this system of positive roots and if $\rho_{M}$ is the half sum
of these positive roots then $\chi_{\xi},_{\nu}=\chi_{\Lambda}$ with $%
\Lambda=\lambda_{\xi}+\rho_{M}+\nu$. This implies the lemma.
\end{proof}

The following result is the reason for the assumption of analyticity.

\begin{lemma}
\label{VanNort-Finite}Let $(\pi,V)$ be an analytic family of admissible $(%
\mathfrak{g},K)$ modules over $Z$. Assume that $z_{0}\in Z$ is such that $%
(\pi_{z_{o}},V)$ is finitely generated. If $T$ is an element of $Z(\mathfrak{%
g})$ there exist analytic functions $a_{0},...,a_{n-1}$ on $Z$ such that if%
\begin{equation*}
f(z,x)=x^{n}+\sum_{j=0}^{n-1}a_{j}(z)x^{j}. 
\end{equation*}
for $z\ $then $f(z,\pi_{z}(T))=0$.
\end{lemma}

\begin{proof}
Let $F$ be a finite number of elements of $\hat{K}$ such that $\pi_{z_{0}}(U(%
\mathfrak{g}))\sum_{\gamma\in F}V(\gamma)=V$. Let $L=$ $\sum_{\gamma\in
F}V(\gamma)$. Then we define the functions $a_{j}$ the by the formula 
\begin{equation*}
f(z,x)=\det\left( xI-\pi_{z}(T\right)
_{|L})=x^{n}+\sum_{j=0}^{n-1}a_{j}(z)x^{j}. 
\end{equation*}
The Cayley-Hamilton theorem implies that $h(z)=T^{n}+%
\sum_{j=0}^{n-1}a_{j}(z)T^{j}\in Z(\mathfrak{g})$ vanishes on $L$. Let $%
\gamma\in\hat{K}$ then there exist $x_{1},...,x_{r}\in U(\mathfrak{g})$ and $%
v_{1},...,v_{r}\in L$ such that $\{\pi_{z_{0}}(x_{i})v_{i}\}_{i=1}^{r}$ is a
basis of $V(\gamma)$. Let $P_{\gamma}$ be the projection onto the $\gamma$%
--isotypic component of $V$. Thus%
\begin{equation*}
(P_{\gamma}\pi_{z}(x_{1})v_{1})\wedge(P_{\gamma}\pi_{z}(x_{2})v_{2})\wedge%
\cdots\wedge(P_{\gamma}\pi_{z}(x_{r})v_{r})\in\wedge^{r}V(\gamma) 
\end{equation*}
(a one dimensional space) is non--zero for $z=z_{0}$. This implies that
there exists an open neighborhood, $U$, of $z_{0}$ in $\Omega$ such that 
\begin{equation*}
P_{\gamma}\pi_{z}(x_{1})v_{1},P_{\gamma}\pi_{z}(x_{2})v_{2},...,P_{\gamma}%
\pi_{z}(x_{r})v_{r}
\end{equation*}
is a basis of $V(\gamma)$ for $z\in U$. That 
\begin{equation*}
h(z)P_{\gamma}\pi_{z}(x_{i})v_{i}=P_{\gamma}\pi_{z}(x_{i})h(z)v_{i}=0 
\end{equation*}
implies that $h(z)V(\gamma)=0$ for $z\in U$. The connectedness of $Z$and the
analyticity imply that $h(z)V(\gamma)=0$ for $z\in Z$. Thus $h(z)=0$ for all 
$z\in Z$. This proves the Lemma.
\end{proof}

If $V$ is a $(\mathfrak{g},K)$--module then set $ch(V)$ equal to the set of $%
\Lambda\in\mathfrak{h}^{\ast}$ such that there exists $v\in V$ with $%
Tv=\chi_{\Lambda}(T)v$ for all $T\in Z(\mathfrak{g})$.

\begin{corollary}
\label{compactness}Keep the notation and assumptions of the previous lemma,
If $\omega\subset Z$ is compact then there exists a compact subset $%
C_{\omega}$ of $\mathfrak{h}^{\ast}$ such that $ch(\pi_{z},V)\subset
C_{\omega}$ for all $z\in\omega$.
\end{corollary}

\begin{proof}
Let $T_{1},...,T_{m}$ be a generating set for $Z(\mathfrak{g})$ and let $%
f_{j}(z,x)$ be the function in the previous lemma corresponding to $T_{j}$.
Then 
\begin{equation*}
f_{j}(z,x)=x^{n_{j}}+\sum_{i=0}^{n_{j}-1}a_{j,i}(z)x^{j}
\end{equation*}
with $a_{j,i}$ analytic in on $Z$. If $\chi_{\Lambda}\in ch(\pi_{z},V)$ then%
\begin{equation*}
\left\vert \chi_{\Lambda}(T_{j})\right\vert \leq\max_{0\leq
i<n_{j}}|a_{j,i}(z)|+1 
\end{equation*}
(c.f. \cite{RRG1-11},7.A.1.3). If $L\subset Z$ is compact then there exists
a constant $r<\infty$ such that $|a_{j,i}(z)|\leq r$ for all $i,j$ and $z\in
L$. This implies the corollary.
\end{proof}

\begin{theorem}
\label{loc-finite}Let $(\pi,V)$ be an analytic family of admissible $(%
\mathfrak{g},K)$ modules over $Z$. Assume that there exists $z_{0}\in Z$
such that $(\pi_{z_{0}},V)$ is finitely generated. If $\omega$ is a compact
subset of $Z$ then there exists $S_{\omega}\subset\hat{K}$ a finite subset
such that if $y\in\omega$ then%
\begin{equation*}
\pi_{y}(U(\mathfrak{g))}\left( \sum_{\gamma\in S_{\omega}}V(\gamma)\right) =V%
\text{.}
\end{equation*}
\end{theorem}

\begin{proof}
Let $C_{\omega}$ as in the above corollary for $\omega$. Let 
\begin{equation*}
X=\{(\xi,\nu)\in\hat{M}\times\mathfrak{a}_{\mathbb{C}}^{\ast}|\chi_{\xi},_{%
\nu}=\chi_{\Lambda},\Lambda\in C_{\omega}\}. 
\end{equation*}
$X$ is compact so there exist $\xi_{1},...,\xi_{r}\in\hat{M}$ and $%
D_{1},...,D_{r}$, compact subsets of $\mathfrak{a}_{\mathbb{C}}^{\ast}$,
such that $X=\cup_{j}\xi_{j}\times D_{j}$. Let $S_{j}\subset\hat{K}$ be the
finite set corresponding to $\xi_{j}\times D_{j}$ in Proposition \ref%
{indfinite}. Set $S_{\omega}=\cup S_{j}$. Let $L_{1}\subset
L_{2}\subset...\subset L_{j}\subset...$ be an exhaustion of the $K$--types
of $V$ with each $L_{j}$ finite.

We will use the notation \thinspace$V_{y}$ for the $(\mathfrak{g},K)$%
--module $(\pi_{y},V)$. Let $y\in C$. Set $W_{j}=\pi_{y}(U(\mathfrak{g))}%
\left( \sum_{\gamma\in L_{j}}V(\gamma)\right) $ then $W_{j}\subset W_{j+1}$
and $\cup W_{j}=V$. Each $W_{j}$ is finitely generated and admissible, hence
of finite length. Therefore $V_{y}$ has a finite composition series 
\begin{equation*}
0=V_{y}^{0}\subset V_{y}^{1}\subset...\subset V_{y}^{N}
\end{equation*}
or a countably infinite composition series%
\begin{equation*}
0=V_{y}^{0}\subset V_{y}^{1}\subset...\subset V_{y}^{n}\subset
V_{y}^{n+1}\subset... 
\end{equation*}
with $V_{y}^{i}/V_{y}^{i-1}$ irreducible. Thus by the dual form of the
subrepresentation theorem there exists for each $i,\xi_{i}\in\hat{M}$ and $%
\nu_{i}\in\mathfrak{a}_{\mathbb{C}}^{\ast}$ so that $V_{y}^{i}/V_{y}^{i}$ is
a quotient of $(\pi_{\xi_{i},\nu_{i}},H^{\xi_{i},\nu_{i}})$. Observe that $%
(\xi_{i},\nu_{i})\in X$. Thus $V_{y}^{i}/V_{y}^{i-1}(\gamma_{i})\neq0$ for
some $\gamma_{i}\in S_{\omega}$. Let $L$ be a quotient module of $V_{y}$.
Then $L=V_{y}/U$ with $U$ a submodule of $V_{y}$. There must be an $i$ such
that $V_{y}^{i}/\left( V_{y}^{i-1}\cap U\right) \neq0$. Let $i$ be minimal
subject to this condition. Then $V_{y}^{i-1}\subset U$. Thus $%
V_{y}^{i}/V_{y}^{i-1}$ is a submodule of $L$. Hence $M(\gamma)\neq0$ for
some $\gamma\in S_{\omega}$. This implies that 
\begin{equation*}
\pi_{y}(U(\mathfrak{g))}\left( \sum_{\gamma\in S_{\omega}}V(\gamma)\right) =V%
\text{.}
\end{equation*}
Indeed, we have shown that 
\begin{equation*}
\left( V_{y}/\pi_{y}(U(\mathfrak{g))}\left( \sum_{\gamma\in
S_{\omega}}V(\gamma)\right) \right) (\gamma)=0,\gamma\in S_{\omega}. 
\end{equation*}
\end{proof}

\begin{corollary}
\label{finitegen}(To the proof) Let $(\pi,V)$ be an analytic family of
finitely generated admissible $(\mathfrak{g},K)$ modules over $Z$ . Let $%
W\subset Z$ be compact . Let for each $z\in W$, $U_{z}$ be a $(\mathfrak{g}%
,K)$--submodule of $V_{z}$. Then there exists a finite subset $F_{W}\subset%
\hat{K}$ such that 
\begin{equation*}
\pi_{z}(U(\mathfrak{g}))\left( \sum_{\gamma\in F_{W}}U_{z}(\gamma)\right)
=U_{z}. 
\end{equation*}
\end{corollary}

\begin{proof}
In the proof of the theorem above all that was used was that the set of
possible infinitesimal characters is compact.
\end{proof}

\section{Continuous and Holomorphic families of smooth Fr\'{e}chet
representations.}

Let $G$ be a reductive group with fixed maximal compact subgroup, $K$. Let $%
\mathcal{F}(G)$ denote the category of smooth Fr\'{e}chet representations
(here smooth means that the map $g\mapsto\pi(g)v$ is $C^{\infty}$).

\begin{definition}
\label{frechet-fam}A continuous family of objects in $\mathcal{F}(G)$ over a
metric space $X$ is a pair $(\pi,V)$ of a Fr\'{e}chet space $V$ and a
continuous map 
\begin{equation*}
\pi:X\times G\rightarrow\mathrm{End}(V) 
\end{equation*}
(here $\mathrm{End}(V)$ is the algebra of continuous operators on $V$ with
the strong topology) such that such that for each $x\in X$, if $%
\pi_{x}(g)=\pi(x,g)$ then $(\pi_{x},V)\in\mathcal{F}(G)$. We will say that
the family has local uniform moderate growth if for each $\omega$ a compact
subset of $X$ and each continuous seminorm on $V,p,$there exists a
continuous seminorm $q_{\omega}$ on $V$ and $r_{\omega}$ such that if $v\in V
$ then 
\begin{equation*}
p(\pi_{x}(g)v)\leq q_{\omega}(v)\left\Vert g\right\Vert ^{r_{\omega}}. 
\end{equation*}
\end{definition}

\begin{proposition}
\label{Cont-diff}If $(\pi,H)$ is a continuous family of Hilbert
representations over the analytic manifold $X$ such that the representations 
$\pi_{x|K}$ are the same for all $x\in X$(we denote this common value by $%
\pi(k))$ and the representations $(d\pi_{x},H_{K}^{\infty})$ form an
analytic family of objects in $\mathcal{H}(\mathfrak{g},K)$,\ then

1. The space of $C^{\infty}$ vectors in $H$ with respect to $\pi_{x}$ is
equal to the space of $C^{\infty}$ vectors of the representation $(\pi,H).$%
of $K$.

2. Assume that for each, $\omega\subset X$, compact, and $u\in U(\mathfrak{g}%
)$ there exist constants $C_{\omega,\upsilon},n_{\omega,u}$ such that%
\begin{equation*}
\left\Vert d\pi_{y}(u)v\right\Vert \leq C_{\omega,u}\left\Vert d\pi
(1+C_{K})^{n_{\omega,u}}v\right\Vert 
\end{equation*}
for $v\in H^{\infty}$. Then $x\mapsto(\pi_{x},H^{\infty})$ is a continuous
family of smooth Fr\'{e}chet representations of local uniform moderate
growth.
\end{proposition}

\begin{proof}
1. follows from Lemma \ref{VanNort-Finite} and Proposition \ref{C-K}.

We now prove 2. To prove the continuity assertion we need to show that if $%
l>0$ and $x_{o}\in X$ then 
\begin{equation*}
\lim_{x\rightarrow x_{o}}\left\Vert
d\pi(1+C_{K})^{l}(\pi_{x}(g)-\pi_{x_{o}}(g))v\right\Vert =0. 
\end{equation*}
Let $\lambda_{\gamma}$ be the eigen-value of $C_{K}$ on $V(\gamma)$. Recall
that if $v\in H^{\infty}=\sum_{\gamma}v_{\gamma}$ with $v_{\gamma}\in
H(\gamma)$ and for each $r$ there exists a constant $C_{r,v}$ such that 
\begin{equation*}
\left\Vert v_{\gamma}\right\Vert \leq C_{v,r}(1+\lambda_{\gamma})^{-r}. 
\end{equation*}
As is well known 
\begin{equation*}
\sum_{\gamma\in\hat{K}}(1+\lambda_{\gamma})^{-r}<\infty 
\end{equation*}
if $r>\frac{\dim T}{2}$ with $T$ is a maximal torus of $K$. Fix $l>0$ and $%
x_{o}$ in $X$. Let $F\subset\widehat{K}$ if $u\in H^{\infty}$ set $%
u(F)=\sum_{\gamma\in F}u_{\gamma}$. If $F^{c}=\widehat{K}-F$, then $%
u=u(F)+u(F^{c}).$ If $u\in H^{\infty}$ then 
\begin{equation*}
d\pi(1+C_{K})^{l}\pi_{x}(g)u=\pi_{x}(g)d\pi_{x}(Ad(g)^{-1}(1+C_{K})^{l})u%
\text{.}
\end{equation*}
Let $z_{1},...,z_{d_{l}}$ be a basis of $U^{l}(\mathfrak{g})$. Then%
\begin{equation*}
Ad(g)^{-1}(1+C_{K})^{l}=\sum a_{i}(g)z_{i}
\end{equation*}
with $a_{i}$real analytic on $G$. Thus%
\begin{equation*}
d\pi(1+C_{K})^{l}\pi_{x}(g)v=\pi_{x}(g)\sum a_{i}(g)d\pi_{x}(z_{i})u. 
\end{equation*}
Note that there exists $C_{1},m$ such that $\left\vert a_{i}(g)\right\vert
\leq C_{1}\left\Vert g\right\Vert ^{m}$ for all $i$. Now fix $x_{o}\in X$
and fix $U$ a neighborhood of $x_{o}$ with compact closure. Then%
\begin{equation*}
\left\Vert \pi_{x}(g)u\right\Vert \leq C_{2}\left\Vert g\right\Vert
^{m_{1}}\left\Vert u\right\Vert ,x\in U,u\in H. 
\end{equation*}
Let $v\in H^{\infty}$. Let $F_{N}=\{\gamma\in\widehat{K}|\lambda_{\gamma}%
\leq N\}$ then $F_{N}$ is a finite set. Let $r=\frac{\dim T}{2}+1.$ Set $%
n=\max _{i}n_{u_{i},\omega}$ with $\omega$ the closure of $U$ and $%
C_{3}=\max C_{u_{i},\omega}$. If $x\in U$%
\begin{equation*}
\left\Vert
d\pi(1+C_{K})^{l}(\pi_{x}(g)-\pi_{x_{o}}(g))v(F_{N}^{c})\right\Vert
^{2}\leq2d_{l}^{2}C_{2}^{2}C_{1}^{2}C_{3}^{2}\left\Vert g\right\Vert
^{m+m_{1}}\sum_{\gamma\notin F_{N}}(1+\lambda_{\gamma})^{2l+2n}\left\Vert
v_{\gamma}\right\Vert ^{2}. 
\end{equation*}
Also%
\begin{equation*}
\left\Vert v_{\gamma}\right\Vert \leq C_{v,m}(1+\lambda_{\gamma})^{-m}, 
\end{equation*}
so%
\begin{equation*}
\sum_{\gamma\notin F_{N}}(1+\lambda_{\gamma})^{2l+2n}\left\Vert v_{\gamma
}\right\Vert ^{2}\leq C_{v,m}\sum_{\gamma\notin F_{N}}(1+\lambda_{\gamma
})^{2l-m}. 
\end{equation*}
Choose $m=2l+r+2n+s$ with $s\geq1$, Then%
\begin{equation*}
\sum_{\gamma\notin F_{N}}(1+\lambda_{\gamma})^{2l+2n}\left\Vert v_{\gamma
}\right\Vert ^{2}\leq N^{-s}C_{v,m}\sum_{\gamma\notin F_{N}}(1+\lambda
_{\gamma})^{-r}\leq N^{-s}C_{v,m}\sum_{\gamma\notin\widehat{K}}(1+\lambda
_{\gamma})^{-r}. 
\end{equation*}
We therefore have ($C_{4}=C_{v,m}\sum_{\gamma\notin\widehat{K}}(1+\lambda
_{\gamma})^{-r}$)%
\begin{equation*}
\left\Vert
d\pi(1+C_{K})^{l}(\pi_{x}(g)-\pi_{x_{o}}(g))v(F_{N}^{c})\right\Vert
^{2}\leq2d_{l}^{2}C_{2}^{2}C_{1}^{2}C_{3}^{2}\left\Vert g\right\Vert
^{m+m_{1}}C_{4}N^{-s}. 
\end{equation*}
Let $\varepsilon>0$ be and let $\omega_{1}$ be a compact subset of $G$.
Choose $N$ so that 
\begin{equation*}
2d_{l}^{2}C_{2}^{2}C_{1}^{2}C_{3}^{2}\left\Vert g\right\Vert
^{m+m_{1}}C_{4}N^{-s}<\frac{\varepsilon^{2}}{4}
\end{equation*}
for $g\in\omega_{1}$. Now if $x\in U$ then 
\begin{equation*}
\left\Vert d\pi(1+C_{K})^{l}(\pi_{x}(g)-\pi_{x_{o}}(g))v\right\Vert
\leq\left\Vert
d\pi(1+C_{K})^{l}(\pi_{x}(g)-\pi_{x_{o}}(g))v(F_{N})\right\Vert 
\end{equation*}%
\begin{equation*}
+\left\Vert
d\pi(1+C_{K})^{l}(\pi_{x}(g)-\pi_{x_{o}}(g))v(F_{N}^{c})\right\Vert
<\left\Vert d\pi(1+C_{K})^{l}(\pi_{x}(g)-\pi_{x_{o}}(g))v(F_{N})\right\Vert +%
\frac{\varepsilon}{2}\text{.}
\end{equation*}
The function of $x$, $\left\Vert
d\pi(1+C_{K})^{l}(\pi_{x}(g)-\pi_{x_{o}}(g))v(F_{N})\right\Vert ^{2}$ is, by
our assumption real analytic in $x,g$ and equal to $0$ at $x=x_{o}$. Hence
there exists a neighborhood, $W$, of $x_{o}$ in $U$ such that if $%
g\in\omega_{1}$ and $x\in W$ then%
\begin{equation*}
\left\Vert d\pi(1+C_{K})^{l}(\pi_{x}(g)-\pi_{x_{o}}(g))v(F_{N})\right\Vert
^{2}<\frac{\varepsilon^{2}}{4}. 
\end{equation*}
This completes the proof of continuity. We leave the condition of uniform
moderated growth to the reader (what is needed is in the above argument).
\end{proof}

\begin{definition}
\label{smoothable}A continuous family of Hilbert representations of $G$, $%
(\pi,H),$ over $X$ will be called smoothable if for each compact subset $%
\omega\subset X,u\in U(\mathfrak{g})$ there exists $C_{\omega,u},n_{\omega
,u}$ such that%
\begin{equation*}
\left\Vert d\pi_{y}(u)v\right\Vert \leq C_{\omega,u}\left\Vert d\pi
(1+C_{K})^{n_{\omega,u}}v\right\Vert 
\end{equation*}
for $y\in\omega,$ $v\in H^{\infty}$.
\end{definition}

\begin{definition}
\label{holomorphicfam}A holomorphic family of objects in $\mathcal{F}(G)$
over the complex manifold $X$ is a continuous family $(\pi,V)$ such that the
map $x\longmapsto\pi_{x}(g)v$ is holomorphic from $X$ to $V$ for all $g\in G$%
,$v\in V$.
\end{definition}

\begin{theorem}
\label{cont-hol}If $(\pi,V)$ is a continuous family of smooth Fr\'{e}chet
representations over the complex manifold $X$ such that $(d\pi,V_{K})$ is a
holomorphic family of objects in $\mathcal{H}(\mathfrak{g},K)$ then $(\pi,V)$
is a holomorphic family of objects in $\mathcal{F}(G)$ over $X$.
\end{theorem}

We will use the next Lemma in the proof.

\begin{lemma}
Let $X$ a complex $n$--manifold, $G$ be connected and let 
\begin{equation*}
f:X\times G\rightarrow\mathbb{C}
\end{equation*}
be continuos and real analytic in $G$. If $zf(x,e)$ is holomorphic in $x\in X
$ for all $z\in U(\mathfrak{g})$ (here $z$ is acting as left invariant
differential operators on the $G$ the second factor) then $f$ is holomorphic
in $X$.
\end{lemma}

\begin{proof}
Let $x\in X$ and let $z_{1},...,z_{n}$ be local coordinates on an open
neighborhood, $U$, \ of $x$ \ in $X$ such that if $\psi=(z_{1},...,z_{n})$
then $\psi(x)=0$ and $\psi(U)\supset\overline{D}^{n}$ with $D$ (resp. $%
\overline{D}$) the (resp. closed) unit disk in $\mathbb{C}$. For simplicity
we may assume that $X=\psi(U)$. Define for $z\in D$%
\begin{equation*}
h(z,g)=\frac{1}{(2\pi i)^{n}}\int_{\left( S^{1}\right) ^{n}}\frac {f(u,g)}{%
\prod(u_{i}-z_{i})}du_{1}\cdots du_{n}. 
\end{equation*}
Then $h(z,g)$ is holomorphic in $z$ on $D$. By our assumption $%
uh(z,e)=uf(z,e)$ for $u\in U(\mathfrak{g}),z\in D$. Since $f$ is analytic in 
$G$ and $G$ is connected $h=f$ on $D$.
\end{proof}

We will now prove the theorem. Let $\lambda\in V^{\prime}$. If $v\in V_{K}$
then the function 
\begin{equation*}
f(x,g)=\lambda(\pi_{x}(g)v) 
\end{equation*}
on $X\times G$ is continuous and real analytic in $G$. Now 
\begin{equation*}
uf(x,e)=\lambda(d\pi_{x}(u)v) 
\end{equation*}
which is holomorphic in $x$. Thus if $v\in V_{K}$ then%
\begin{equation*}
x\mapsto\lambda(\pi_{x}(g)v) 
\end{equation*}
is holomorphic in $x$. Let $x\in X$ and $U,\psi$, etc. be as in the previous
lemma. Set%
\begin{equation*}
h(z,v)=\frac{1}{(2\pi i)^{n}}\int_{\left( S^{1}\right) ^{n}}\frac {%
\lambda(\pi_{u}(g)v)}{\prod(u_{i}-z_{i})}du_{1}\cdots du_{n}. 
\end{equation*}
for $v\in V,z\in D^{n}$. Then $h(z,v)$ is holomorphic in $z$ and continuous
in $v$. Furthermore, the first part of this proof showed that $%
h(z,v)=\lambda (\pi_{z}(g)v)$ if $v\in V_{K}$. Since, $V_{K}$ is dense in $V$
this implies that%
\begin{equation*}
z\mapsto\lambda(\pi_{z}(g)v) 
\end{equation*}
is holomorphic in $z$. Grothendieck \cite{Groth-an} \ has shown that a
weakly holomorphic map of a complex manifold to a Fr\'{e}chet space is
strongly holomorphic thus completing the proof.

\section{Functorial properties of Hilbert families}

In this section we will analyze Hilbert globalizations of subfamilies and
quotient families of Harish-Chandra modules.

\begin{lemma}
\label{orthhbases}Let $(\tau,V)$ be a finite dimensional continuous
representation of $K$ and let $X$ be a locally compact metric space (resp.
an analytic manifold). If $u\in X$ let $\left\langle ...,...\right\rangle
_{u}$ be an inner product on $V$ such that $\tau(k)$ acts unitarily with
respect to $\left\langle ...,...\right\rangle _{u}$ for $k\in K$ and such
that the map $u\longmapsto\left\langle v,w\right\rangle _{u}$ is continuous
(resp. real analytic) for all $v,w\in V$. Then there exists, for each $u$ an
ordered orthonormal basis of $V,$ $e_{1}(u),...,e_{n}(u)$ such that the map $%
u\longmapsto e_{i}(u)$ is continuous (resp. real analytic) and the matrix of 
$\tau(k)$ with respect to $e_{1}(u),...,e_{n}(u)$ is independent of $u$.
Furthermore, if $X$ is compact and contractible and $(\sigma.W)$ is a finite
dimensional continuous representation of $K$ and $u\longmapsto B(u)\in 
\mathrm{Hom}_{K}(V,W)$ is continuous and surjective for $u\in X$ then $%
e_{1}(u),...,e_{r}(u)$ with $r=\dim V-\dim W$ can be taken in $\ker B(u)$.
\end{lemma}

\begin{proof}
Fix an inner product, $(...,...)$, on $V$ such that $\tau$ is unitary. Then
there exists a positive definite Hermitian operator (with respect to $\left(
...,...\right) $), $A(u)$ such that $\left\langle v,w\right\rangle
_{u}=(A(u)v,w),v,w\in V$ and $A(u)$ is continuous (resp. real analytic) in $u
$. Note that 
\begin{equation*}
\tau(k)^{-1}A(u)\tau(k)=A(u),u\in X,k\in K. 
\end{equation*}
Set $S(u)=A(u)^{\frac{1}{2}}$ then $\left\langle v,w\right\rangle
_{u}=(S(u)v,S(u)w)$. Thus if $T(u)=S(u)^{-\frac{1}{2}}$ then $\tau
(k)T(u)=T(u)\tau(k),k\in K,u\longmapsto T(u)$ is continuous (resp. real
analytic) and 
\begin{equation*}
\left\langle T(u)v,T(u)w\right\rangle _{u}=\left( v,w\right) ,v,w\in V. 
\end{equation*}
Let $e_{1},...,e_{n}$ be an (ordered) orthonormal basis of $V$ with respect
to $(...,...)$ then $e_{1}(u)=T(u)e_{1},...,e_{n}(u)=T(u)e_{n}$ is an
orthonormal basis of $V$ with respect to $\left\langle ...,...\right\rangle
_{u}$ .If $\tau(k)e_{i}=\sum k_{ji}e_{j}$ then%
\begin{equation*}
\tau(k)e_{i}(u)=\tau(k)T(u)e_{i}=T(u)\tau(k)e_{i}=\sum k_{ji}T(u)e_{j}. 
\end{equation*}
To prove the second assertion note that $u\rightarrow\ker B(u)$ is a $K$%
--vector bundle over $X$ (see the lemma below). Since $X$ compact and
contractible the bundle is a trivial $K$--vector bundle (\cite{Atiyah} Lemma
1.6.4). Thus there is a representation $(\mu,Z)$ of $K$ and $u\longmapsto
L(u)\in Hom_{K}(Z,V)$ continuous such that $L(u)Z=\ker B(u)$ and $L(u)$ is
injective. Notice that $B(u):\ker B(u)^{\perp}\rightarrow W$ is a $K$%
--module isomorphism. Now pull back the inner product $\left\langle
_{...,...}\right\rangle _{u}$ to $Z$ using $L(u)$ getting a $K$--invariant
inner product, $(...,,,,)_{u}$, on $Z$ and push the inner product to $W$
getting \ a $K$--invariant inner product $(...,...)_{u}^{1}$ on $W$ Now
apply the first part of the lemma to get an orthonormal basis $%
f_{1}(u),...,f_{r}(u)$ of $Z$ with respect to $(...,...)_{u}$ and an
orthonormal basis $f_{r+1}(u),...,f_{n}(u)$ ($n=\dim V$) with respect to $%
(...,...)_{u}^{1}$ such that the matrices of the action of $K$ with respect
to each of these bases is constant. Take $e_{i}(u)=L(u)f_{i}(u)$ for $%
i=1,...,r$ and $e_{i}(u)=\left( B(u)_{|_{\ker B(u)^{\perp}}}\right)
^{-1}f_{i}(u)$ for $i=r+1,...,n$.
\end{proof}

\begin{lemma}
Let $V$ and $W$ be finite dimensional, continuous $K$--modules and assume
that for $x\in X$, $B(x)\in\mathrm{Hom}_{K}(V,W)$ is surjective and the map $%
x\mapsto B(x)$ is continuous. Then $x\mapsto\ker B(x)$ is a $K$--vector
bundle over $X$.
\end{lemma}

\begin{proof}
Let $x_{o}\in X$ and let $M\subset V$ be a $K$--invariant subspace of $V$
such that $B(x_{o})$ is a $K$--isomorphism of $M$ onto $W$. Then there
exists $U\subset X$ an open neighborhood of $x_{o}$ such that $B(u)_{|M}$ is
invertible for $u\in U$. Set $S(u)=\left( B(u)_{|Z}\right) ^{-1}$ on $B(u)V,$
for $u\in U$. If $v\in\ker B(x_{0})$ and if $u\in U$ then%
\begin{equation*}
B(u)v=B(u)S(u)B(u)v 
\end{equation*}
so%
\begin{equation*}
B(u)(I-S(u)B(u))v=0. 
\end{equation*}
Thus $I-S(u)B(u)$ maps $\ker B(x_{o})$ to $\ker B(u)$ for $u\in U$. This map
is the identity for $u=x_{o}$, so it is a $K$--isomorphism for $u\in
U_{1}\subset U$ with $U_{1}$ open in $U$.
\end{proof}

\begin{proposition}
\label{Hilb-access}If $(\sigma,V)$ is a continuous family of admissible $(%
\mathfrak{g},K)$--modules over a metric space $X$, if $(\mu,H)$ is an
smoothable (see Definition \ref{accessible}) continuous family of admissible
Hilbert representations of $G$ based on $X$ and if 
\begin{equation*}
T:(\sigma,V)\rightarrow(d\mu,H_{K}) 
\end{equation*}
is a continuous family of injective $(\mathfrak{g},K)$--module homomorpisms
then there exists $(\lambda,W)$ a smoothable continuous family of Hilbert
representations of $G$ based on $X$ such that $(d\lambda,W_{K}^{\infty})$
and $(\sigma,V)$ are isomorphic as continuous families and a continuous
family of injections of $(\lambda,W)$ into $(\mu,H)$.
\end{proposition}

\begin{proof}
If $x\in X$ then set $\left\langle ...,...\right\rangle _{x}=T_{x}^{\ast
}\left\langle ...,...\right\rangle $. If $\gamma\in\widehat{K},x\in X$ let $%
e_{1}^{\gamma}(x),...,e_{n_{\gamma}}^{\gamma}(x)$ be the orthonormal basis
as in Lemma \ref{orthhbases} corresponding to $\left\langle
...,...\right\rangle _{x}$. Then $\{e_{i}^{\gamma}(x)\}$ is an orthonormal
basis of $V$. Set 
\begin{equation*}
f_{i}^{\gamma}(x)=T_{x}(e_{i}^{\gamma}(x)) 
\end{equation*}
then $\{f_{i}^{\gamma}(x)\}$ is an orthonormal basis of $T_{x}V$ for $x\in X$%
. If $v\in H$ then set 
\begin{equation*}
P_{\gamma}(x)v=\sum_{i=1}^{\nu_{\gamma}}\left\langle v,f_{i}^{\gamma
}(x)\right\rangle f_{i}^{\gamma}(x). 
\end{equation*}
The map 
\begin{equation*}
x\mapsto P_{\gamma}(x) 
\end{equation*}
is Strongly continuous from $X$ to $\mathrm{Hom}_{K}(H,H(\gamma))$(the
continuous $K$ \ homomorphisms). Define 
\begin{equation*}
P(x)v=\sum P_{\gamma}(x)v. 
\end{equation*}
Then $P(x)$ is the orthogonal projection of $H$ onto the closure of $T_{x}V$
in $H$. Thus in particular $\left\Vert P(x)\right\Vert =1$. We assert that 
\begin{equation*}
x\mapsto P(x) 
\end{equation*}
is strongly \ continuous from $X$ to the bounded operators on $H$. To this
end, let $v\in H$ \ \ be a unit vector and $x_{o}\in X$. Let $\varepsilon>0$
be given and let $F\subset\widehat{K}$ be such that 
\begin{equation*}
\left\Vert \sum_{\gamma\notin F}v(\gamma)\right\Vert <\frac{\varepsilon}{4}
\end{equation*}
then since 
\begin{equation*}
P(x)\sum_{\gamma\in F}v(\gamma)=\sum_{\gamma\in F}P_{\gamma}(x)v(\gamma) 
\end{equation*}
there exists an open neighborhood, $U$, of $x_{o}$ in $X$ such that%
\begin{equation*}
\left\Vert \left( P(x)-P(x_{o})\right) \sum_{\gamma\in F}v(\gamma
)\right\Vert <\frac{\varepsilon}{2}. 
\end{equation*}
Thus 
\begin{equation*}
\left\Vert \left( P(x)-P(x_{o})\right) v\right\Vert \leq\left\Vert \left(
P(x)-P(x_{o})\right) \sum_{\gamma\in F}v(\gamma)\right\Vert +\left\Vert
\left( P(x)-P(x_{o})\right) \sum_{\gamma\notin F}v(\gamma)\right\Vert 
\end{equation*}%
\begin{equation*}
\leq\left\Vert \left( P(x)-P(x_{o})\right) \sum_{\gamma\in F}v(\gamma
)\right\Vert +2\left\Vert \sum_{\gamma\notin F}v(\gamma)\right\Vert
<\varepsilon\text{.}
\end{equation*}

Let $\nu_{x}(g)$ be the action of $G$ on $P(x)H$. Define $%
L(x,y):P(y)H\rightarrow P(x)H$ 
\begin{equation*}
L(x,y)f_{i}^{\gamma}(y)=f_{i}^{\gamma}(x). 
\end{equation*}
Then $L(x,y)$ is a unitary operator and a $K$--module equivalence.
Furthermore,%
\begin{equation*}
x,y\mapsto L(x,y)P(y) 
\end{equation*}
is strongly continuous (use a slight modification of the argument for the
strong continuity of $P(x)$). Fix $x_{o}\in X$ and set $W=L(x_{o})H$. Set $%
\lambda_{x}(g)=L(x_{o},x)\nu_{x}(g)L(x,x_{o})$.To complete the proof we need
to show that $(\mu,W)$ is smoothable. Let $\omega$ be a compact subset of $X$
and $u\in U(\mathfrak{g})$ then there exist $C_{\omega}$ and $n_{\omega}$
such that if $v\in H^{\infty}$ then%
\begin{equation*}
\left\Vert d\mu_{x}(u)v\right\Vert \leq C_{\omega}\left\Vert
d\mu_{x}(1+C_{K})^{n_{\omega}}v\right\Vert . 
\end{equation*}
Now, if $v\in W^{\infty}$ then $v=P_{x_{o}}w$ with $w\in H^{\infty}$. So%
\begin{equation*}
\left\Vert d\lambda_{x}(u)v\right\Vert =\left\Vert L(x_{o},x)d\nu
_{x}(u)L(x,x_{o})P_{x_{0}}w\right\Vert 
\end{equation*}%
\begin{equation*}
=\left\Vert d\nu_{x}(u)L(x,x_{o})P_{x_{0}}w\right\Vert \leq C_{\omega
}\left\Vert d\nu_{x}(1+C_{K})^{n_{\omega}}L(x,x_{o})P_{x_{0}}w\right\Vert 
\end{equation*}%
\begin{equation*}
=C_{\omega}\left\Vert
L(x_{o},x)d\nu_{x}(1+C_{K})^{n_{\omega}}L(x,x_{o})P_{x_{0}}w\right\Vert 
\end{equation*}%
\begin{equation*}
=C_{\omega}\left\Vert
d\lambda_{x}(1+C_{K})^{n_{\omega}}P_{x_{0}}w\right\Vert 
\end{equation*}
\end{proof}

Similarly we can prove

\begin{proposition}
\label{hilb-surj}If $(\sigma,V)$ is a continuous family of admissible $(%
\mathfrak{g},K)$--modules over a compact contractible metric space $X$, $%
(\mu,H)$ is a continuous family of admissible Hilbert representations of $G$
based on $X$ and if 
\begin{equation*}
T:(d\mu,H_{K})\rightarrow(\sigma,V) 
\end{equation*}
is a continuous family of surjective $(\mathfrak{g},K)$--module homomorpisms
then there exists $(\lambda,W)$ an acceptable, continuous family of Hilbert
representations of $G$ based on $X$ such that $(d\lambda,W_{K}^{\infty})$
and $(\sigma,V)$ are isomorphic as continuous families and a continuous
family of surjections of $(\mu,H)$ into $(\lambda,W)$. Furthermore, if $%
(\mu,H)$ is smoothable the so is $(\lambda,W)$.
\end{proposition}

\begin{proof}
The proof follows the same lines as the previous theorem. Let for each $x\in
X$, 
\begin{equation*}
B_{\gamma}(x)=T_{x}|_{H(\gamma)}. 
\end{equation*}
Let $r_{\gamma}=\dim V(\gamma),m_{\gamma}=\dim H(\gamma)$. Then Lemma \ref%
{orthhbases} implies that for each $x\in X$ there exists an orthonormal
basis of $H(\gamma),\{e_{i}^{\gamma}(x)\}$ with respect to the inner
product, $\left\langle ...,...\right\rangle $ on $H$ such that $%
\{e_{i}^{\gamma }(x)\}_{i>r_{\gamma}}$ is a basis of $\ker B_{\gamma}(x)$
and $x\mapsto$ $e_{\gamma}^{i}(x)$ is continuous. Let $f_{i}^{\gamma}(x)=B_{%
\gamma}(x)e_{i}^{\gamma}(x)$ and for each $x\in X$ define an inner product, $%
\left\langle ...,...\right\rangle _{x}$ on $V$ by declaring that $\left\{
f_{i}^{\gamma}(x)\right\} $ is an orthonormal basis. Set 
\begin{equation*}
P_{\gamma}(v)=\sum_{i=1}^{r_{\gamma}}\left\langle v,e_{i}^{\gamma
}(x)\right\rangle e_{i}^{\gamma}(x) 
\end{equation*}
and $P(x)=\sum P_{\gamma}(x)$. Essentially the same argument as in the
preceding theorem shows that map $x\mapsto P(x)$ is strongly continuous.
Also, $T_{x}:P(x)H_{K}\rightarrow V$ is unitary relative to $\left\langle
...,...\right\rangle _{x}$ and an equivalence of representations of $K$. Let 
$H_{1,x}$ be the closure in $H$ of $\ker T_{x}$ for $x\in X$. Then, as a
Hilbert space, $H/H_{1,x}=P(x)H$. Since $\ker T_{x}$ consists of analytic
vectors $H_{1,x}$ is $G$--invariant, This defines a Hilbert representation, $%
\gamma_{x}$, on $P(x)H$. Which in turn defines a Hilbert representation, $%
\nu_{x}$, of $G$ on the Hilbert space completion of $V$, $Z_{x}$. Let $%
L(x,y):Z_{y}\rightarrow Z_{x}$ be defined by $L(x,y)f_{i}^{%
\gamma}(y)=f_{i}^{\gamma}(x)$. Then $L(x,y)$ defines a unitary $K$%
--isomorphism of $(\nu_{y}|_{K},Z_{y})$ with $(\nu_{x}|_{K},Z_{x})$. Fix $%
x_{o}$ in $X$ and let $W=Z_{x_{o}}$ and $\lambda_{x}(g)=L(x_{o},x)%
\nu_{x}(g)L(x,x_{o}).$ Then, as in the above theorem, we have defined a
Hilbert family globalizing $(\sigma.V)$.

We now assume that $(\mu,H)$ is smoothable. If $v\in P_{x}H^{\infty}$ and
then%
\begin{equation*}
d\mu_{x}(u)v=d\gamma_{x}(u)v+(I-P_{x})d\mu_{x}(u)v. 
\end{equation*}
Thus if $\omega$ is a compact subset of $X$ then%
\begin{equation*}
\left\Vert d\gamma_{x}(u)v\right\Vert \leq\left\Vert d\mu_{x}(u)v\right\Vert 
\end{equation*}%
\begin{equation*}
\leq C_{u,\omega}\left\Vert d\mu_{x}(1+C_{K})^{l}v\right\Vert =C_{u,\omega
}\left\Vert d\gamma_{x}(1+C_{K})^{l}v\right\Vert . 
\end{equation*}
We reinterpret the family $(\lambda,W)$. Let $M(x,y):P_{y}H\rightarrow P_{x}H
$ be given by 
\begin{equation*}
M(x,y)e_{i}^{\gamma}(y)=e_{i}^{\gamma}(x). 
\end{equation*}
Fix $x_{o}\in X$. Then the family can be defined as $%
\delta_{x}(g)=M(x_{o},x)\gamma_{x}(g)M(x,x_{o})v$ for $v\in P_{x_{o}}H.$
Setting $U=P_{x_{o}}H$ then $(\delta,U)$ is an isomorphic continuous family
to $(\lambda,W)$. We show that this family is smoothable let $u\in U(%
\mathfrak{g})$. Then if $v\in U$ and $\omega$ is a compact subset of $X$
then 
\begin{equation*}
\left\Vert d\delta_{x}(u)v\right\Vert =\left\Vert M(x_{o},x)d\gamma
_{x}(u)M(x,x_{o})v\right\Vert =\left\Vert
d\gamma_{x}(u)M(x,x_{o})v\right\Vert 
\end{equation*}%
\begin{equation*}
\leq C_{\upsilon,\omega}\left\Vert
d\gamma_{x}(1+C_{K})^{l}M(x,x_{o})v\right\Vert =C_{u,\omega}\left\Vert
M(x_{o},x)d\gamma_{x}(1+C_{K})^{l}M(x,x_{o})v\right\Vert 
\end{equation*}%
\begin{equation*}
=C_{u,\omega}\left\Vert d\delta_{x}((1+C_{K})^{l}v\right\Vert . 
\end{equation*}
\end{proof}

\end{document}